\documentclass[12pt]{amsart}
\usepackage{amsmath}
\usepackage{amsfonts}
\usepackage{amssymb}
\usepackage[mathscr]{eucal}
\usepackage{graphicx}
\makeatletter
\def\numberwithin#1#2{\@ifundefined{c@#1}{\@nocnterrr}{%
  \@ifundefined{c@#2}{\@nocnterr}{%
  \@addtoreset{#1}{#2}%
  \toks@\expandafter\expandafter\expandafter{\csname the#1\endcsname}%
  \expandafter\xdef\csname the#1\endcsname
    {\expandafter\noexpand\csname the#2\endcsname
     .\the\toks@}}}}
\makeatother
\numberwithin{equation}{section}

\newtheorem{theorem}{Theorem}
\numberwithin{theorem}{section}

\newtheorem{corollary}[theorem]{Corollary}

\newtheorem{example}[theorem]{Example}
\newtheorem{lemma}[theorem]{Lemma}

\newtheorem{proposition}[theorem]{Proposition}
\newtheorem{remark}[theorem]{Remark}

\newenvironment{rmk}{\begin{remark} \em}{\end{remark}}

\newcommand{\tr}{\operatorname{tr}}
\newcommand{\Rc}{\operatorname{Rc}}
\newcommand{\Int}{\operatorname{int}}

\newcommand{\SO}{{\rm SO}}
\newcommand{\SU}{{\rm SU}}

\newcommand{\U}{{\rm U}}

\newcommand{\tf}{\mbox{${\mathfrak t}$}}

\newcommand{\C}{\mbox{${\mathbb C}$}}

\newcommand{\PP}{\mbox{${\mathbb P}$}}

\newcommand{\R}{\mbox{${\mathbb R}$}}
\newcommand{\Z}{\mbox{${\mathbb Z}$}}

\begin{document}
\title{Ancient Ricci Flow Solutions on Bundles}

\author{Peng Lu}
\address{Department of Mathematics, University of Oregon, Eugene, OR 97403-1222}
\email{penglu@uoregon.edu}
\thanks{}

\author{Y. K. Wang}
\address{Department of Mathematics and Statistics, McMaster
     University, Hamilton, Ontario, L8S 4K1, CANADA}
\email{wang@mcmaster.ca}
\thanks{P.L. is partially supported by Simons Foundation through Collaboration Grant 229727.
Y.K.W. is partially supported by NSERC Grant No. OPG0009421.}

\date{\small revised \today}

\begin{abstract}
We generalize the circle bundle examples of ancient solutions of the Ricci flow discovered
by Bakas, Kong, and Ni to a class of principal torus bundles over an arbitrary finite product of
Fano K\"ahler-Einstein manifolds studied by Wang and Ziller in the context of Einstein geometry.
As a result, continuous families of $\kappa$-collapsed and $\kappa$-noncollapsed ancient solutions
of type I are obtained on circle bundles for all odd dimensions $\geq 7$.
In dimension $7$ such examples moreover exist on pairs of homeomorphic but not diffeomorphic
manifolds.   Continuous families of $\kappa$-collapsed ancient solutions
of type I are also obtained on torus bundles for all dimensions $\geq 8$.

\smallskip


\end{abstract}

\maketitle

\section{\bf Introduction}

An ancient solution $g(t)$ of the Ricci flow in this paper is one that exists on a half infinite interval of
the type $(-\infty, T)$ for some finite time $T$ and with $g(t)$ complete for each $t$.
These solutions are important because they arise naturally when finite time
singular solutions of the flow are blown up. Rigidity phenomena for ancient solutions have been established by
L. Ni (\cite{Ni09}), and by S. Brendle, G. Huisken, and C. Sinestrari (\cite{BHS11}) under certain non-negativity
conditions on the full curvature tensor.  For dimension two, classification theorems were achieved by
P. Daskalopoulos, R. Hamilton, and N. Sesum (\cite{DHS12}) for closed surfaces and by
R. Hamilton (\cite{Ha95}), S.C. Chu (\cite{Chu07}), and  P. Daskalopoulos and N. Sesum (\cite{DS06})
for complete non-compact surfaces.

Examples of ancient solutions in higher dimensions do exist. The first examples (in dimension $3$) were due to
Fateev \cite{Fa96} and Perelman \cite{Pe02}. Other examples in dimension $3$ result from the work of X.D. Cao,
J. Guckenheimer, and L. Saloff-Coste (\cite{CS09}, \cite{CGS09}) on the backwards behavior of the Ricci
flow on locally homogeneous $3$-manifolds. Fateev's example was then generalized by
I. Bakas, S. L. Kong, and L. Ni (\cite{BKN12}) to odd-dimensional spheres, complex projective spaces,
and the twistor spaces of compact quaternion-K\"{a}hler manifolds.  By studying a system of two
$\operatorname{ODE}$s J. Lauret (\cite{La13}) gave explicit examples of $2$-parameter
families of left-invariant ancient solutions on any compact simple Lie group except $Sp(2k + 1)$.
Further homogeneous examples can be found in \cite{Bu14} and \cite{BLS16}. Finally, we mention
two inhomogeneous examples in dimension $4$: a compact ancient solution due to S. Brendle and N.
Kapouleas (\cite{BK14}) which is constructed by desingularizing an orbifold quotient of the flat torus, and
a non-compact inhomogeneous example due to R. Takahashi (\cite{Ta14}) which has the Euclidean
Schwarzschild metric as backwards limit.

In this article we shall be concerned with the construction of ancient solutions on closed
manifolds which are not highly symmetric except in special cases. Specifically, we shall
generalize the circle bundle examples in \cite[Theorem 5.1]{BKN12} to a class of principal
torus bundles over an arbitrary finite product of Fano KE (short for K\"ahler-Einstein)  manifolds.
One may view our work
here as a dynamic version of the work in \cite{WZ90}, although the aim, of course, is not
to reconstruct the Einstein metrics obtained in that work via the Ricci flow.

To describe our main results, let $P_Q$ denote a principal $r$-torus bundle with Euler classes
$Q$ over a product of $m$ arbitrary KE manifolds $(M_i, g_i)$ with positive
first Chern class. As in \cite{WZ90} the integral cohomology classes in $Q$ are taken to be allowable
rational linear combinations of the first Chern classes of $M_i$ subject to certain non-degeneracy
conditions explained in detail in \S 2. We consider the class of connection-type metrics on $P_Q$
determined by an arbitrary left-invariant metric on the torus fibres, arbitrary scalings of the
KE metrics $g_i$ by constants, and the harmonic representatives of the Euler classes in $Q$
with respect to the base metric chosen. This class of metrics is preserved by the Ricci flow,
and the Ricci flow equation simplifies to a system of nonlinear ODEs, which we analyse in some detail
to deduce

\vskip 0.3cm

\noindent{\bf Main Theorem.} {\em With assumptions as described above, on each total space $P_Q$

\vskip .1cm

\noindent $($a$)$ there is an $(m+r(r+1)/2-1)$-parameter family of ancient
solutions of type I for which the
time-rescaled backwards limit is a multiple of the product of the K\"ahler-Einstein metrics
on the base $M_1 \times \cdots \times M_m$. These solutions are $\kappa$-collapsed.

\noindent  $($b$)$ When $r=1$ and $m \geq 2$, there is a further $(m-1)$-parameter
family of ancient solutions of type I that is $\kappa$-noncollapsed. Their time-rescaled
backwards limit is a multiple of the Einstein metric on $P_Q$ found in \cite{WZ90}. }

\vskip 0.3cm

Detailed descriptions of the asymptotics and curvature properties of the above ancient
solutions can be found in Theorem \ref{curv--toruscase}, Theorems
\ref{thm flow near origin}--\ref{prop asym beha at infty}, and Theorem \ref{prop asym beha xi at infty}.

We also investigate the forward limits of our ancient solutions in some special situations
when $r=1$. First, we show that among the ancient solutions in part (a) of the above theorem
on a circle bundle there is a unique solution which develops a type I singularity in finite time
and whose time-rescaled forward limit is a multiple of the Einstein metric on the circle bundle
(see Theorem \ref{lem growth a bi}). Second, when $m=2$ we show that
the ancient solutions in part (b) of the above theorem develop a type I singularity in finite time
and their time-rescaled forward limit
is the product of a circle bundle over one of the base factors with a Euclidean space of the same
dimension as the other base factor. Furthermore, this limit is $\kappa$-noncollapsed for some
$\kappa > 0$ (see Theorem \ref{thm Curv property m=2 Omega 1 2 a b, tau}).

It should be noted that none of the above compact ancient solutions can occur as  finite-time
singularity models of the Ricci flow.
This is because such models must, on the one hand, be $\kappa$-noncollapsed at all
scales by the celebrated work of Perelman (\cite{Pe02}), and, on the other hand, be shrinking Ricci
solitons when the model is compact by a result of Z.L. Zhang (\cite{Zh07}).

However, there appears to be some independent interest in ancient solutions of the Ricci flow
beyond singularity analysis, as exemplified by the classification work of
Daskalopoulos-Hamilton-Sesum \cite{DHS12} and the work of physicists on the Renormalization
Group flow \cite{Fr85}. From this vantage point, our works shows that, at least among torus bundles,
families of non-solitonic ancient solutions of the Ricci flow (collapsed or not) with positive Ricci
curvature do exist in abundance. The torus bundles we consider here are moreover quite diverse in their
topological properties, as was shown in \cite{WZ90}.  In particular, they include manifolds
with infinitely many homotopy types as well as infinitely many homeomorphism types within a fixed
integral cohomology type. Furthermore,  at least in dimension $7$ (and most probably  for
infinitely many dimensions) they include pairs of homeomorphic but not diffeomorphic manifolds
\cite{KS88} (see Remark \ref{remark top type cir} for further discussion).

In the context of Einstein geometry, our investigation also uncovers some interesting
phenomena. As mentioned above, in the circle bundle case, there is a unique ancient solution
that develops a type I finite time
singularity whose time-rescaled blow-up limit is the Einstein metric constructed in  \cite{WZ90}.
One may regard this solution as giving a curve {\em that satisfies a geometric PDE} \lq\lq
connecting\rq\rq the product KE metric on the base (Theorem \ref{prop asym beha at infty}(i)),
which is a manifold of one
dimension less, to an Einstein metric on the circle bundle (Theorem \ref{lem growth a bi}(iii)).
A similar curve is constructed in Theorem \ref{thm Curv property m=2 Omega 1 2 a b, tau}(i).
Such solutions are especially intriguing in cases for which the total spaces are diffeomorphic
but the Einstein metrics belong to distinct path components of the Einstein moduli space.
For example, this happens already for dimension $5$: $S^2 \times S^3$ is a circle bundle
over $S^2 \times S^2$ in countably infinitely many ways, and the unit volume Einstein metrics
have Ricci curvature converging to zero.

In analysing the Ricci flow restricted to connection-type metrics, we are dealing with
an $\operatorname{ODE}$ system which is in general not bounded in the number of
dependent variables. After a suitable change of both independent and dependent variables,
this system becomes a polynomial system for which we find suitable bounded sets, monotonic
quantities, and differential inequalities. These are used to show that the solutions stay in these sets,
which in turn allows us to deduce that the solutions exist for all times approaching $-\infty$.
The geometric properties of the solutions are then established from various estimates derived
from the $\operatorname{ODE}$ system. In order to derive part (b) of our main result, we
also need to study the eigenvalues of the linearization of our polynomial system at the fixed
point corresponding to the Einstein metric mentioned in part (b) of the Main Theorem.

Finally, we want to point out an interesting correspondence between the Riemannian ancient
solutions of the Ricci flow we found and certain ancient solutions as well as immortal
solutions of the Ricci flow equations for appropriate pseudo-Riemannian metrics. More precisely,
if we replace some of the Fano KE factors in the base of our circle bundles by KE manifolds with
the same dimension but with {\em negative scalar curvature}, and use the ansatz
(\ref{eq family of Riem metric}) to define the pseudo-Riemannian metrics, we obtain both ancient
and immortal solutions for the corresponding pseudo-Riemannian Ricci flow on the resulting
bundles. For further details see Remark \ref{rk psudo Riem RF ancient}.

\vskip .1cm
The following is an outline of the rest of this article. In \S \ref{sec 2 RF eq}
 we derive the Ricci flow equation for metrics on torus bundles over a product of Fano KE manifolds
which satisfy the ansatz (\ref{eq family of Riem metric}). In \S \ref{sec 3 S1 bundle} we show the existence of
ancient solutions on the circle bundles and study their geometric behavior as time $t \rightarrow -\infty$.
We also show the existence of one particular ancient solution whose finite time singularity model
is the Einstein metric found in \cite[Theorem 1.4]{WZ90}. In \S \ref{subsec noncollapsing anc sol m2}
we study in detail some ancient solutions whose singularity model at time $t=-\infty$ is this Einstein metric.
In \S \ref{sec 4 torus bundles}  we prove the general existence of ancient solutions on torus bundles
and discuss their geometric behavior as time $t \rightarrow -\infty$.
In the appendix we compute the eigenvalues of a matrix related to the linearization of
our polynomial system at the Einstein metric.

In a follow-up article we will study ancient solutions of the Ricci flow on certain
 $(\SO(3) \times \cdots \times \SO(3))/\Delta \SO(3)$ fibre bundles
over an arbitrary finite product of compact quaternionic K\"ahler manifolds.


\section{\bf Ricci flow equations for connection-type metrics on certain torus bundles} \label{sec 2 RF eq}

Let $(M_i^{n_i}, g_i), \, i=1, \cdots, m$, be KE manifolds  of complex dimension $n_i$ with Ricci tensor
$\operatorname{Rc} (g_i)=p_i g_i$ for some nonzero $p_i \in \Z$.
Let $\omega_i$ be the K\"{a}hler form associated with metric $g_i$.
By a suitable homothetic change of $g_i$, we may make the following {\bf normalization assumption}
throughout this article: the cohomology class $\frac{1}{2\pi}[\omega_i]$ is integral and
represents a torsion free indivisible class in $H^2(M_i, \Z)$.

Let $\pi:P \rightarrow M_1 \times \cdots \times M_m$ be a smooth principal $r$-torus bundle.
If we identify the Lie algebra $\tf$ of the torus $T^r$ with $\R^r$ by choosing a basis
$\{e_\alpha; \alpha=1, \cdots, r\}$ for $\tf$ associated with a fixed decomposition $T^r
=S^1 \times \cdots \times S^1$, then we get an induced action of $S^1 \times \cdots \times S^1$ on $P$,
and the principal torus bundles over
$M_1 \times \cdots \times M_m$ are classified by $r$ Euler classes $\chi_\alpha, \, \alpha =1, \cdots,r$, in
$H^2(M_1 \times \cdots \times M_m; \,\Z) $. Here $\chi_\alpha$ is the Euler class of the orientable
circle bundle $P/T^{r-1} \rightarrow M_1 \times \cdots \times M_m$, where $T^{r-1}
\subset T^r$ is the subtorus with the $i$th $S^1$ factor deleted.
Let $A$ be an automorphism of $T^r$. $A$ induces a new decomposition of $T^r=S^1 \times \cdots \times
S^1$, and a new basis $\tilde{e}_\alpha =A_{\alpha \beta}e_\beta$, where matrix $A=(A_{\alpha \beta}
)_{r \times r}$ is an integral matrix with $\det (A) =\pm 1$.
The new induced action of $S^1 \times \cdots \times S^1$ on $P$ has associated Euler classes $\tilde{\chi}
_\alpha = (A^T)^{-1}_{\alpha \beta} \chi_\beta$.
 In the rest of the article we will fix a decomposition $T^r
=S^1 \times \cdots \times S^1$ and hence we will  have a fixed basis $\{e_\alpha; \alpha=1, \cdots, r\}$ for
the Lie algebra $\tf$ of $T^r$.

 We are interested in those principal torus
bundles whose Euler classes $\chi_\alpha, \, \alpha =1, \cdots,r$, are given by $\frac{1}{2\pi}
\sum_i q_{\alpha i} [\omega_i]$, where $q_{\alpha i}$ are integers. Let $Q$ denote the $r \times m$
 matrix $(q_{\alpha i})$.
Note that if we change the decomposition  $T^r =S^1 \times \cdots \times S^1$ by $A$ as discussed above,
then the new $\tilde{Q} =(A^T)^{-1} Q$, hence (a) the ranks of $Q$ and $\tilde{Q}$ are the same and (b) the
numbers of columns of zeros in $Q$ and in $\tilde{Q}$ are the same.
Also note that if we switch the order of factors in product $M_1 \times \cdots \times M_m$,
the new $Q$ matrix still satisfies
the properties (a) and (b). Hence we can make the following definition.
We call a principal $T^r$ bundle $P_Q$ \textbf{non-degenerate} if the associated
matrix $Q$ has rank $r$ and has no column of zeros. This assumption implies $r \leq m$.
When $r=m$, by applying a suitable automorphism (equivalently changing the decomposition of
the torus into a product of circles), one sees that $P$ is a finite quotient of a product
of circle bundles over $M_i$. Hence the non-trivial cases begin with $m=2$ for circle bundles
and $r=2, m=3$ for torus bundles, and the corresponding minimum dimensions of $P$ are
respectively $5$ and $8$.

On a torus bundle $P_Q$, using Hodge theory, we may choose an $\R^r$-valued principal connection
$\sigma =(\sigma_1, \cdots, \sigma_r)$ such that the curvature form $F^{\alpha}$ of
$\sigma_\alpha$ is given by
\begin{equation}
F^\alpha = d \sigma_{\alpha} = \sum_{i=1}^m \, q_{\alpha i} \,\omega_i.   \label{eq curv form T r connect}
\end{equation}
Throughout this article we will use the \textbf{convention} $\tau= -t$ as the backwards time.
Let $(h_{\alpha \beta}(\tau))  \doteqdot (h(e_\alpha,e_\beta))$ be a $1$-parameter family of
left-invariant metrics on $T^r$. Since $T^r$ is commutative, these metrics are automatically
also right-invariant. Using $\sigma$, we now obtain a family of Riemannian metrics on $P_Q$ given by
\begin{equation}
g_{h, \vec{b}} (\tau) = \sum_{\alpha, \beta=1}^r h_{\alpha \beta}(\tau)\, \sigma_\alpha( \cdot) \otimes
\sigma_\beta(\cdot) + \sum_{i=1}^m \,b_i(\tau)\, \pi_i^* g_i, \label{eq family of Riem metric}
\end{equation}
where $\pi_i=\tilde{\pi}_i \circ \pi$, $\tilde{\pi_i}: M_1 \times \cdots \times M_m \rightarrow M_i$
is the standard projection, and $b_i(\tau) > 0$.
We shall often abuse notation by referring to $\pi_i^*g_i$ as $g_i$.
The above family of metrics makes $\pi$ into a Riemannian submersion with totally geodesic fibres
for each fixed $\tau$.

The components of the Ricci tensor of $g_{h, \vec{b}} (\tau)$  are given by
(see \cite[p.224]{WZ90} and \cite[\S2.2]{BKN12})
\begin{subequations}
\begin{align}
& R(g_{h, \vec{b}}(\tau))_{\alpha\beta} = \sum_{\tilde{\alpha},\tilde{\beta}
=1}^r \sum_{i=1}^m \frac{1}{2}\, n_i q_{\tilde{\alpha} i}q_{\tilde{\beta}
i} \frac{h_{\alpha \tilde{\alpha}}(\tau) h_{\beta \tilde{\beta}}(\tau)}{b_i^2(\tau)},
\label{Ricci-toral-components} \\
& R(g_{h, \vec{b}} (\tau))_{k \alpha } = 0,  \label{eq Ricci curv formula}\\
& R(g_{h, \vec{b}} (\tau))_{kl} =  \operatorname{Rc}(g_i)_{kl} \frac{1}{b_i(\tau)}
-\sum_{\tilde{\alpha},\tilde{\beta}
=1}^r \frac{1}{2} \,  q_{\tilde{\alpha} i}q_{\tilde{\beta} i} (g_i)_{kl}\frac{
h_{\tilde{\alpha} \tilde{\beta}}(\tau)} {b_i^2(\tau)},
\end{align}
\end{subequations}
where in the last equation $k$ and $l$ are the indices used for an orthonormal frame
$\{e_k^{(i)}\}$ for the factor $M_i$ with respect to the metric $b_i(\tau)\,g_i$.
Notice that equation (\ref{eq Ricci curv formula}) automatically holds for all $\tau$ as a
result of our choice of connection because the harmonicity of the curvature form for
all values of $\tau$ is precisely the Yang-Mills condition for the connection metric
$g_{h, \vec{b}} (\tau)$ (see Proposition 9.36 of \cite{Bes87}).
 Hence the Ricci tensor can be written as
\begin{align}
\operatorname{Rc}(g_{h, \vec{b}}(\tau)) = & \sum_{\alpha,\beta, \tilde{\alpha},\tilde{\beta}
=1}^r \sum_{i=1}^m \frac{1}{2}\, n_i q_{\tilde{\alpha} i} q_{\tilde{\beta} i} \frac{h_{\alpha
\tilde{\alpha}}(\tau) h_{\beta \tilde{\beta}}(\tau) }
{b_i^2(\tau)} \sigma_\alpha \otimes \sigma_\beta  \notag \\
& +\sum_{i=1}^m \left ( \frac{ p_i }{b_i(\tau)} - \sum_{\tilde{\alpha},\tilde{\beta}=1}^r \frac{1}{2}
 q_{\tilde{\alpha} i}q_{\tilde{\beta} i} \frac{ h_{\tilde{\alpha} \tilde{\beta}}(\tau)}
{b_i^2(\tau)} \right) b_i(\tau) g_i. \label{eq Ricci curv circle bundle}
\end{align}

Since we choose $\tau = -t$, a solution to the Ricci flow for time $t \in [0, T)$
corresponds to a solution of the backwards Ricci flow for $\tau \in(-T, 0]$.
With this convention, the backwards Ricci flow of $g_{h, \vec{b}}(\tau)$ is given by the
$\operatorname{ODE}$ system
\begin{subequations}
\begin{align}
&\frac{d h_{\alpha \beta}}{d\tau} = \sum_{\tilde{\alpha},\tilde{\beta}
=1}^r \sum_{i=1}^m  n_i \,q_{\tilde{\alpha} i} q_{\tilde{\beta} i} \frac{h_{\alpha
\tilde{\alpha}}(\tau) h_{\beta \tilde{\beta}}(\tau) }{b_i^2(\tau)}, \,\,\,\, 1 \leq \alpha,
\beta \leq r, \label{eq RF torus r A1} \\
& \frac{d b_i}{d\tau} =  2p_i  - \sum_{\tilde{\alpha},\tilde{\beta}=1}^r
 q_{\tilde{\alpha} i}q_{\tilde{\beta} i} \frac{ h_{\tilde{\alpha} \tilde{\beta}}(\tau)}
 {b_i(\tau)}, \,\,\, 1 \leq i \leq m. \label{eq RF torus r A2}
\end{align}
\end{subequations}
If $g_{h, \vec{b}}(\tau)$ has long time existence, then $g_{h, \vec{b}}(\tau)$
is an ancient solution of the Ricci flow since our bundles are compact.

\begin{rmk} \label{degeneracy}
Note that if we require the diagonal condition
$h_{\alpha \beta}(\tau) =a_\alpha (\tau) \delta_{\alpha \beta}$ after
we fix an appropriate decomposition $T^r =S^1 \times \cdots \times S^1$,
 we get from (\ref{eq RF torus r A1})  that for $\alpha \neq \beta$
\[
\sum_{i=1}^m  n_i q_{ \alpha i} q_{\beta i} \frac{a_{\alpha}(\tau) a_{\beta}(\tau)
}{b_i^2(\tau)} =0 \qquad \text{ for all } \tau.
\]
Naively one would expect that $n_i q_{\alpha i} q_{\beta i}=0$ for any $i$ and $\alpha \neq \beta$.
This then implies following conditions on $q_{\alpha i}$.
After a permutation of the indices $i$, $q_{11}, \cdots, q_{1 i_1}$ would be
nonzero while $q_{1 i} =0$ for all $i> i_1$, and for each $\alpha=2, \cdots, r$, we can
arrange for $q_{\alpha i_{\alpha -1}+1},  \cdots, q_{\alpha i_{\alpha} }$ to be nonzero and
$q_{\alpha i} =0$ for all the other $i$. For each $\alpha =1, \cdots, r$ let $P_\alpha \rightarrow
M_{i_{\alpha-1}+1} \times \cdots M_{i_\alpha}$ be the principal circle bundle defined by first Chern
class $\sum_{i= i_{\alpha -1}+1}^{ i_\alpha} q_{\alpha i} \omega_i$. Let $\hat{\pi}_\alpha: M_1
\times \cdots \times M_m \rightarrow  M_{i_{\alpha-1}+1} \times \cdots \times M_{i_\alpha}$ be the
natural projection map. Then the principal torus bundle $P_Q$ is isomorphic to the product bundle
$\hat{\pi}_1^* P_1 \times \cdots \times \hat{\pi}_r^* P_r$, and the Ricci flow on $P_Q$ decouples to
the Ricci flows on the circle bundles $P_\alpha$.
Recall that  by our definition the circle bundle  $P_\alpha$ is non-degenerate if each $q_{\alpha i}$
is nonzero for $i=i_{\alpha -1}+1, \cdots, i_\alpha$.
This is the case for which we will prove the existence of ancient solutions in next section.
\end{rmk}

\begin{rmk} As input data for the torus bundles under consideration we can certainly take $(M_i, g_i)$
to be compact homogeneous KE spaces. These are precisely the coadjoint orbits of
the compact semisimple Lie groups with the induced metric, and the resulting torus bundles
will actually be homogeneous (see Proposition 3.1 in \cite{WZ90}). The resulting Ricci flows
are therefore homogeneous, as the Ricci flow preserves isometries. For this special case our
analysis partially overlaps with that done in \cite{Bu14}, \cite{Buz2}, \cite{Bo15}, and \cite{BLS16}.

On the other hand, this is not the generic situation and there are many concrete examples
of Fano KE manifolds whose isometry groups are far from being transitive. An interesting family
of cohomogeneity $3$ are the small deformations of the Mukai-Umemura $3$-fold \cite{Do08}.
Examples with at most a finite automorphism group include $\C\PP^2$ with $k$
generic points blown up ($4 \leq k \leq 8$). These surfaces even have a positive-dimensional
moduli space of complex structures, so that the KE metrics come in continuous families. See
Remark \ref{cohomogeneity} for more comments on the isometry groups of our ancient flow solutions.
\end{rmk}


\section{\bf Ancient solutions  of Ricci flow on circle bundles} \label{sec 3 S1 bundle}

In this section we consider the $r=1$ case of the backwards Ricci flow (\ref{eq RF torus r A1}) and
(\ref{eq RF torus r A2}) and study the existence and asymptotic geometric properties of the ancient
solutions.

 Using notations $h_{11}(\tau)\doteqdot a(\tau)$ and $q_{1i} \doteqdot q_i$,
the family of Riemannian metrics $g_{h, \vec{b}} (\tau)$ on the circle bundles $P_Q$ becomes
\begin{equation}
g_{a, \vec{b}} (\tau) = a(\tau) \, \sigma( \cdot) \otimes \sigma (\cdot)
+ \sum_{i=1}^m \, b_i(\tau) \, g_i. \label{eq family of Riem metric r=1}
\end{equation}
The backwards Ricci flow equation then simplifies to the system
\begin{subequations}
\begin{align}
& \frac{d a}{d\tau} =\sum_{i=1}^m n_i q_i^2 \frac{a^2(\tau)}{b_i^2(\tau)},  \label{eq RF torus r=1 A1}\\
& \frac{d b_i}{d\tau} = 2 p_i -q_i^2\frac{a(\tau)}{b_i(\tau)}, \,\,\, 1 \leq i
 \leq m. \label{eq RF torus r=1 A2}
\end{align}
\end{subequations}
Unless otherwise stated, in this section we shall assume the non-degeneracy condition that
$q_1, \cdots,q_m$ are nonzero and
the Fano condition that $p_1,\cdots, p_m$ are positive.

\subsection{\bf  A polynomial system, its linearization and monotonicity property}
\label{subsec 3.1 preparation}
To analyze system (\ref{eq RF torus r=1 A1}) and (\ref{eq RF torus r=1 A2}),
we introduce new dependent variables
\begin{equation}
 Y_i \doteqdot \frac{a}{b_i}, \,\,\,\, 1 \leq i \leq m,   \label{Y-defn}
\end{equation}
and let $Y = (Y_1, \cdots,Y_m)$ be the corresponding vector in $\R^m$.
We also introduce a new independent variable $u$ by the relation
\begin{equation} \label{u-defn}
     u= u(\tau)\doteqdot \int_{0}^\tau \frac{1}{a(\zeta)}\,\, d \zeta.
\end{equation}
Since we shall need $a(\tau)$ to be positive, $u$ increases with $\tau$.
We define the following useful functions
\begin{subequations}
\begin{eqnarray}
E(Y) &\doteqdot& \sum_{i=1}^m n_i q_i^2 Y_i^2,    \label{E-defn} \\
F_i(Y) &\doteqdot& 2p_iY_i -q_i^2Y_i^2 - E(Y), \quad 1 \leq i \leq m.  \label{Fi-defn}
\end{eqnarray}
\end{subequations}

One may now express the system (\ref{eq RF torus r=1 A1}) and (\ref{eq RF torus r=1 A2})
in terms of the new variables:
\begin{subequations}
\begin{align}
& \frac{d \ln a}{du} = E(Y),  \label{eq RF torus r=1 A1a}  \\
& \frac{d Y_i}{d u} = - Y_i F_i(Y), \quad 1 \leq i \leq m. \label{eq RF torus r=1 A2a}
\end{align}
\end{subequations}
Below we will abuse notations by writing $a(u) =a(\tau(u))$ and $Y_i(u)=Y_i(\tau(u))$.
Notice that the subsystem (\ref{eq RF torus r=1 A2a}) determines $Y$ as a function of $u$, and
(\ref{eq RF torus r=1 A1a}) allows us to recover $a$ as a function of $u$. More precisely, we have
\begin{equation} \label{formula-au}
   a(u) = a(0) \exp \left(  \int_0^u  E(Y(\zeta))\, d\zeta \right).
\end{equation}
Finally, $\tau$ can be recovered from integrating $d\tau = a(u) du$, and $b_i(\tau)$ from $a(u)/Y_i(u)$
and the relation between $u$ and $\tau$.
Hence given a solution $Y(u)$ of (\ref{eq RF torus r=1 A2a}) it gives rise to a solution $g_{a,\vec{b}}(\tau)$
of the backwards Ricci flow.
  Therefore for circle bundles it suffices for us to focus on
the system (\ref{eq RF torus r=1 A2a}).

\begin{rmk} \label{Einstein-point}
There is only one point $\xi=(\xi_1, \cdots, \xi_m) $ which satisfies
$F_{i} (\xi) = 0$ and $\xi_i > 0$ for all $i=1, \cdots, m$. This actually corresponds
to the Einstein metric of connection type and of positive scalar curvature found in \cite{WZ90},
which is unique up to homothety among metrics of connection type. To see this, note that $F_i (Y) = 0$
for all $i$ means that $Y_i(p_i - \frac{1}{2} q_i^2 Y_i) = \frac{1}{2} E(Y)$. We may set the right
hand side to be $\Lambda a$ where $\Lambda$ and $a$ are positive constants. With $b_i \doteqdot a/Y_i$,
upon comparison with (1.5) and (1.6) in \cite{WZ90}, we obtain the desired conclusion.
\end{rmk}

We define the nonempty compact convex subset
\begin{equation}
 \Omega_+ \doteqdot \{Y \in \mathbb{R}_{\geq 0}^m: \, F_i (Y) \geq 0 \text{ for each
} i=1,\cdots, m \},  \label{eq def Omega +}
\end{equation}
where the boundary $\partial \Omega_+$ is the union of portions of the
level hypersurfaces $F_{i} (Y) = 0$. We define also the subset
\begin{equation*}
\Omega_- \doteqdot \{Y \in \mathbb{R}_{\geq 0}^m: \, F_i (Y) \leq 0
\text{ for each } i=1,\cdots, m \}.
\end{equation*}
More generally, let $\Theta$ be the collection of nonempty proper subsets of $\{ 1,2, \cdots, m \}$.
For each $\theta \in \Theta$ we define the subsets
\begin{equation*}
\Omega_{\theta} \doteqdot \{Y \in \mathbb{R}_{\geq 0}^m: \,  F_i (Y) \geq 0 \text{ for each
} i \in \theta \text{ and } F_j (Y) \leq 0  \text{ for each
} j \notin \theta \}.
\end{equation*}
In short we denote $\Omega_{\{i \}}$ by $ \Omega_i$.

The next lemma describes the constant solutions of (\ref{eq RF torus r=1 A2a}).
Let $G(Y)$ denote the vector field $ (Y_1 F_1(Y), \cdots, Y_mF_m(Y))$,
 the negative of the vector field in (\ref{eq RF torus r=1 A2a}).

\begin{lemma}\label{lem Y i fixed point}
 The zeros of the vector field $G(Y)$ are

\smallskip
\noindent $($i$)$ the origin,

\noindent $($ii$)$ $\{ v_{\theta}, \theta \in \Theta \}$, where $v_\theta$ is the only solution which satisfies
the equations $F_i(Y)=0$  and $Y_i>0$ for $i \in \theta$, and $Y_j=0$ for $j \notin \theta$.
 Note that $v_\theta \in \Omega_\theta \setminus \Omega_+$ and
$v_{\theta} \in \Omega_-$.

\noindent $($iii$)$ the Einstein point $\xi $ described in Remark \ref{Einstein-point}.
Note that $\xi$ belongs to $\Omega_+$, $\Omega_-$, and each $\Omega_\theta$.
\end{lemma}

\begin{proof}
The zeros are given by $Y_1F_1(Y) =\cdots=Y_mF_m(Y) =0$. After a  permutation of the
indices, we may assume that $Y_1  \neq0, \cdots, Y_k \neq 0$, $Y_{k+1} = \cdots = Y_m = 0$,
and $F_{1}(Y)= \cdots= F_k(Y) = 0$, for some $k \in \{0, 1, \cdots, m \}$.

If $k=0$, we are in Case (i), if $k=m$ we are in Case (iii) and $\xi \in \partial \Omega_+$,
 otherwise we are in one of
the situations in Case (ii) where $\theta =\{1, \cdots, k \}$.
Note that $F_i(Y) \geq 0$ implies that $Y_i \geq 0$.

Now suppose $k \in \{ 1,\cdots, m-1 \}$ and $\theta = \{1, \cdots, k \}$  in Case (ii). Then
$F_j(Y) = - \sum_{l=1}^k n_l q_l^2Y_l^2 <0$ for any $j \geq k+1$.
So the solution $Y$ lies outside $\Omega_+$ and inside $\Omega_{\theta} \cap \Omega_-$.
To see the existence and the uniqueness of the solution, we notice that
the issue is to solve
\[
 2p_i Y_i -   q_i^2 Y_i^2  =  E((Y_{1}, \cdots, Y_k)), \, i=1, \cdots, k
\]
with all $Y_i >0$. This follows from the existence and
the uniqueness mentioned in Remark \ref{Einstein-point}.
\end{proof}

To facilitate the study of the global dynamical behavior of the nonlinear system
(\ref{eq RF torus r=1 A2a}),  we need to consider the linearization of vector field $G(Y)$
at each of the fixed points $0, v_\theta, \xi$ and a monotonicity formula for solution $Y(u)$.

\begin{lemma}\label{lem newly minted Feb 20}
Let matrix $\mathcal{L}_{Y}$ denote the linearization of $G(Y)$ at $Y$.
Then

\smallskip
\noindent  $($i$)$ $\mathcal{L}_{\xi}$ is diagonalizable with one negative and $m-1$ positive eigenvalues.
An eigenvector of the negative eigenvalue can be chosen to have all positive entries.

\noindent  $($ii$)$ For each nontrivial subset $\theta \subset \{ 1, \cdots, m\}$, $\mathcal{L}_{v_\theta}$
has $m- |\theta|+1$ negative eigenvalues and $|\theta|-1$  positive eigenvalues.

\noindent  $($iii$)$  $\mathcal{L}_{0}$ is zero matrix $0_{m \times m}$.
\end{lemma}

\begin{proof}
$($i$)$ Note that  $\mathcal{L}_{\xi}$ is the matrix
\begin{equation*}
 \left [ \begin{array}{cccc}
2\xi_1(p_1 -(n_1+1)q_1^2\xi_1) & -2n_2q_2^2\xi_1 \xi_2    & \cdots  & -2 n_m q_m^2\xi_1 \xi_m  \\
-2 n_1q_1^2\xi_2 \xi_1  & 2\xi_2(p_2 -(n_2+1)q_2^2\xi_2)  & \cdots  & -2 n_mq_m^2 \xi_2 \xi_m  \\
\vdots & \vdots & \vdots  & \vdots \\
-2 n_1q_1^2 \xi_m \xi_1 & -2n_2q_2^2\xi_m \xi_2   & \cdots & 2\xi_m(p_m -(n_m+1)q_m^2\xi_m )
 \end{array} \right ].
\end{equation*}
  Define the matrix
\begin{equation*}
 \beta \doteqdot \left [ \begin{array}{cccc}
 (2n_1+1)q_1^2\xi_1^2 & 2n_2q_2^2\xi_1 \xi_2    & \cdots  & 2 n_m q_m^2\xi_1 \xi_m  \\
2n_1q_1^2 \xi_2 \xi_1  & (2n_2+1)q_2^2\xi_2^2  & \cdots  & 2n_mq_m^2 \xi_2 \xi_m    \\
\vdots & \vdots & \vdots  & \vdots \\
2n_1q_1^2\xi_m \xi_1   & 2n_2q_2^2\xi_m \xi_2   & \cdots & (2n_m+1)q_m^2\xi_m^2
 \end{array} \right ].
\end{equation*}
Then we have $\mathcal{L}_{\xi} = E(\xi)I_{m \times m} -\beta$.
Define the diagonal matrix $D_{\xi} \doteqdot [ \xi_1, \cdots,\xi_m]$ and the
 matrix of positive entries $\alpha \doteqdot D_{\xi}^{-1} \beta D_{\xi}$.
We have
\begin{equation*}
 \alpha \doteqdot \left [ \begin{array}{cccc}
 (2n_1+1)q_1^2\xi_1^2 & 2 n_2q_2^2 \xi_2^2   & \cdots  & 2 n_m q_m^2 \xi_m^2  \\
2 n_1q_1^2 \xi_1^2 & (2n_2+1)q_2^2\xi_2^2  & \cdots  & 2  n_mq_m^2 \xi_m^2  \\
\vdots & \vdots & \vdots  & \vdots \\
2  n_1q_1^2 \xi_1^2 & 2 n_2q_2^2 \xi_2^2 & \cdots & (2n_m+1)q_m^2\xi_m^2
 \end{array} \right ].
\end{equation*}
Note that the matrix $\alpha$ and $\beta$ have the same eigenvalues.
 We denote by $\lambda_1(\alpha)$ the largest eigenvalue of $\alpha$.
This is known in the literature as the  \textbf{Perron-Frobenius eigenvalue}.

Let $a_i =  2n_iq_i^2\xi_i^2$ and $\epsilon_i =\frac{1}{2n_i}$.
By Lemma \ref{lem eigenvalue ai matrix and diag}(ii)  in the appendix, the
matrix $\alpha$ is diagonalizable with eigenvalues  $\lambda_i(\alpha), \, i=1, \cdots, m$.
 Since the eigenvalues of a matrix depend continuously on the matrix,
we may still apply Lemma \ref{lem eigenvalue ai matrix and diag}(i) to estimate $\lambda_i(\alpha)$
even when not all $\epsilon_i a_i$ are distinct, except that we need to change $<$'s to $\leq$'s in
(\ref{eq lambda two est getby}).

 Note that each of the row sums in $\alpha$ is greater than $2 E(\xi)$.
Hence the smallest  eigenvalue of $\mathcal{L}_{\xi}$ satisfies
$E(\xi) - \lambda_1(\beta) < - E(\xi)<0$.
The corresponding eigenvector of $\mathcal{L}_{\xi}$ is the eigenvector
of $\beta$ corresponding to $\lambda_1(\beta)=\lambda_1(\alpha)$,
which has positive entries by Lemma \ref{lem eigenvalue ai matrix and diag}(i)
and the fact that the entries of $D_{\xi}$ are positive.
The other eigenvalues of $\alpha$ satisfy
\[
 \min_{i=1, \cdots,m} \{q_i^2 \xi_i^2 \} =
 \min_{i=1, \cdots,m} \{\epsilon_i a_i \} \leq  \lambda_j(\alpha) \leq \max_{i=1, \cdots,m}
\{\epsilon_i a_i \} = \max_{i=1, \cdots,m} \{q_i^2 \xi_i^2 \}
\]
where  $j=2, \cdots, m$.
Hence the corresponding eigenvalues of $\mathcal{L}_{\xi}$ satisfy
\[
\lambda_j(\mathcal{L}_{\xi}) = E(\xi)- \lambda_j(\beta) \geq  E(\xi)-  \max_{i=1, \cdots,m}
\{q_i^2 \xi_i^2 \} >0.
\]

$($ii$)$ Without loss of generality we may assume $\theta =\{1, \cdots, k \}$ with $1 \leq k \leq m-1$.
Then $\mathcal{L}_{v_\theta}$ is given by
\begin{eqnarray*}
 \left [ \begin{array}{cccccc}
2Y_1(p_1 -(n_1+1)q_1^2Y_1) & \cdots & -2n_kq_k^2 Y_1 Y_k &0    & \cdots  & 0  \\
\vdots & \vdots & \vdots & \vdots & \vdots & \vdots   \\
-2 n_1q_1^2Y_k Y_1  & \cdots & 2Y_k(p_k -(n_k+1)q_k^2Y_k) & 0  & \cdots  & 0  \\
0 & \cdots & 0  & F_{k+1}(Y)  & \cdots & 0  \\
\vdots & \vdots & \vdots  & \vdots  & \vdots & \vdots  \\
0 & \cdots & 0   & 0 & \cdots & F_m(Y)
 \end{array} \right ]
\end{eqnarray*}
where $Y=v_\theta$.
Note that this is a block diagonalized matrix since the submatrix in the upper-left corner
is the linearization of the truncated system $(Y_1F_1(Y), \cdots, Y_kF_k(Y) )$ at the
corresponding Einstein point. By part (i) the submatrix has one negative eigenvalue and $k-1$
positive eigenvalues. Since $F_{j}(v_{\theta}) <0$ for $j =k+1, \cdots, m$,
 (ii) is now proved.

 $($iii$)$  This is obvious.
\end{proof}

Now we discuss a monotonicity formula for solution $Y(u)$ of (\ref{eq RF torus r=1 A2a}).
Let $\mathcal{W}(g,f,\tau)$ be Perelman's entropy functional.
Let  $\mu(g, \tau) \doteqdot \inf _{f} \mathcal{W}(g,f,\tau)$ and $\nu(g) \doteqdot \inf_{\tau >0}
\mu(g, \tau)$. It is well-known that $\nu(g)$ is a monotone quantity under the Ricci flow.
For the metrics $g_{a, \vec{b}}$ in (\ref{eq family of Riem metric r=1}) we define
\begin{equation}
\bar{\lambda}(g_{a, \vec{b}}) \doteqdot \bar{\lambda} (Y) \doteqdot
\left ( \prod_{i=1}^mY_i^{- \frac{2n_i}{n}} \right )
\cdot \sum_{i=1}^m  (2n_i p_iY_i  -\frac{1}{2} n_i q_i^2 Y_i^2 )  \label{eq bar lambda def}
\end{equation}
where $n =1+ \sum_{i=1}^m 2n_i$ and once again $Y_i =\frac{a}{b_i} >0$.
Note that in our setting $\bar{\lambda} (Y) $ is a smooth function of $Y \in \mathbb{R}^m_{> 0}$.
If we assume that the minimizing function $f$ of  $ \inf_{f} \mathcal{W}(g_{a, \vec{b}},f,\tau) $ is a constant,
then a simple calculation shows that $\nu(g_{a, \vec{b}})$ is a linear function
of $\ln (\bar{\lambda}(g_{a, \vec{b}}))$
which would imply the monotonicity of $\bar{\lambda}(g_{a, \vec{b}})$ under the Ricci flow.
Actually we can show this monotonicity directly.

\begin{lemma} \label{lem bar lambda monotone}
Under $\operatorname{ODE}$  $($\ref{eq RF torus r=1 A2a}$)$ we have
\begin{align}
& \left ( \prod_{i=1}^mY_i^{- \frac{2n_i}{n}} \right )^{-1} \frac{d \bar{\lambda}(Y(u))
}{du}  \notag  \\
\leq & - \left ( \frac{1}{\sqrt{n}} \left ( \sum_{i=1}^m n_iY_i^2 (2 p_i
 - q_i^2Y_i)^2 \right  )^{1/2} -\sqrt{\frac{n-1}{2n}} E(Y) \right )^2 \leq 0.  \label{eq deriv bar lambda}
\end{align}
Furthermore $\frac{d \bar{\lambda}(Y(u)) }{du} =0$ for some $u=u_0$
if and only if $F_i(Y(u_0)) =0$ for each $i$, i.e., $Y(u) = \xi$ for all $u$.
\end{lemma}

\begin{proof}
We compute
\begin{align*}
& \left ( \prod_{i=1}^mY_i^{- \frac{2n_i}{n}} \right )^{-1} \frac{d \bar{\lambda}(Y(u)) }{du} \\
=& \left (\frac{4}{n} \sum_{i=1}^m  n_ip_iY_i  - \frac{n+1}{n} E(Y) \right )
\sum_{i=1}^m  \,(2n_i p_iY_i  -\frac{1}{2} n_i q_i^2 Y_i^2 )  \\
& + \sum_{i=1}^m  \,(- 2n_i p_iY_iF_i(Y)  + n_i q_i^2 Y_i^2 F_i(Y) )  \\
=& \frac{8}{n} \left ( \sum_{i=1}^m  n_ip_iY_i  \right)^2- \frac{2}{n} E(Y) \sum_{i=1}^m  n_ip_iY_i
 - \frac{n+1}{n}E(Y) \sum_{i=1}^m  (2n_i p_iY_i  -\frac{1}{2} n_i q_i^2 Y_i^2 )  \\
 & -\sum_{i=1}^m  (2n_i p_iY_i  - n_i q_i^2 Y_i^2 ) (2p_i Y_i- q_i^2Y_i^2 -E(Y))  \\
 =& \frac{8}{n} \left ( \sum_{i=1}^m  n_ip_iY_i  \right)^2 - \frac{4}{n} E(Y) \sum_{i=1}^m  n_ip_iY_i
 -\frac{n-1}{2n}E(Y)^2 - \sum_{i=1}^m  n_i (2p_iY_i -q_i^2Y_i^2 )^2 .
 \end{align*}
 We have proved
 \begin{align}
& \left ( \prod_{i=1}^mY_i^{- \frac{2n_i}{n}} \right )^{-1} \frac{d \bar{\lambda}(Y(u)) }{du}
 \notag \\
 =& \frac{2}{n}  \left (2 \sum_{i=1}^m   n_i p_i Y_i   - \frac{1}{2} E(Y)\right)^2
 -\frac{1}{2}E(Y)^2 - \sum_{i=1}^m  n_iY_i^2 (2p_i -q_i^2Y_i )^2 . \label{eq d du bar lambda}
\end{align}

Note that by the Cauchy-Schwartz inequality we have
\begin{align*}
& \left| 2 \sum_{i=1}^m  n_i p_i Y_i   - \frac{1}{2} E(Y)\right| \leq \left|\sum_{i=1}^m  n_i Y_i (2 p_i
 - q_i^2Y_i)\right| + \frac{1}{2} E(Y) \\
\leq &  \left ( \sum_{i=1}^m  n_i  \right )^{1/2} \left (\sum_{i=1}^m n_iY_i^2 (2 p_i
 - q_i^2Y_i)^2 \right  )^{1/2} + \frac{1}{2} E(Y).
\end{align*}
Hence
\begin{align*}
  \left ( 2 \sum_{i=1}^m  n_i p_i Y_i   - \frac{1}{2} E(Y)\right)^2
 \leq & \,\, \left(\frac{n-1}{2} \right) \sum_{i=1}^m n_iY_i^2 (2 p_i  - q_i^2Y_i)^2 +\frac{1}{4}E(Y)^2 \\
& +\sqrt{\frac{n-1}{2}} \, E(Y)  \left (\sum_{i=1}^m n_iY_i^2 (2 p_i
 - q_i^2Y_i)^2 \right  )^{1/2}.
\end{align*}
Combining this inequality with (\ref{eq d du bar lambda}) we get (\ref{eq deriv bar lambda}).

To have  $\frac{d \bar{\lambda}(Y(u)) }{du} =0$ at some $u_0$, we need
\begin{equation}
\sum_{i=1}^m n_iY_i^2(u_0)(2p_i -q_i^2 Y_i(u_0))^2 = \frac{n-1}{2} E(Y(u_0))^2
 \label{eq tem mar 19A}
\end{equation}
and for some constant $c >0$
\begin{align}
& (\sqrt{n_1}\,Y_1(u_0)(2p_1-q_1^2Y_1 (u_0)), \cdots, \sqrt{n_m}\,
 Y_m (u_0)(2p_m-q_m^2Y_m(u_0)))  \notag \\
= & c(\sqrt{n_1}, \cdots, \sqrt{n_m}), \label{eq tem mar 19B}
\end{align}
the latter condition owing to the use of the Cauchy-Schwartz inequality.
Equation (\ref{eq tem mar 19B}) gives $2p_iY_i(u_0)-q_i^2Y_i^2(u_0) =c$.
Substituting this into (\ref{eq tem mar 19A}) we get $c^2 =E(Y(u_0))^2$. Hence we have proved
$F_i(Y(u_0)) =2p_iY_i(u_0)-q_i^2Y_i(u_0)^2 -E(Y(u_0)) =c -c=0$.
\end{proof}


\subsection{\bf Existence of ancient solutions on circle bundles}
\label{subsec exist long S1 m Omega +}

In this subsection we show the existence of the ancient
solutions of the Ricci flow on circle bundles and  derive the asymptotic properties of the
corresponding functions $a$ and $b_i$.
We begin with the global dynamical behavior of the system (\ref{eq RF torus r=1 A2a}).

\begin{lemma} \label{lem non ancient or ancient sol}
 $($i$)$ If $Y(0) \in \mathbb{R}^m_{\geq 0}$, then the solution $Y(u) \in \mathbb{R}^m_{\geq 0}$
  for all $u \in [0, u_*)$ where $u_*$ is the maximal time of existence, and
 if $Y(u) \in \mathbb{R}^m_{> 0}$ holds for some $u=u_0$ then it holds for all $u$.

\noindent $($ii$)$ If for some $Y(0) \in \mathbb{R}^m_{> 0}$
there are $i$ and $u_1 \in [0,u_*)$ such that $Y_i(u_1) > \frac{2p_i}{(n_i+1)q_i^2}$,
then $u_* < \infty$ and $Y(u)$ approaches  infinity  as $u \rightarrow u_*^-$.
The corresponding solution of the Ricci flow is not ancient.

\noindent $($iii$)$ If $Y(u) \in \mathbb{R}^m_{\geq 0}$ is a solution satisfying
 $Y_i(u) \leq \frac{2p_i}{(n_i+1)q_i^2}$ for all $u \geq 0$,
 then $Y(u)$ is defined on $[0, \infty)$ .
 Furthermore either $\lim_{u \rightarrow \infty} Y(u)$ equals to $\xi$
 and the corresponding solution of the Ricci flow is ancient,
 or $Y(u)$ approaches some coordinate hyperplane $Y_i=0$ as $u \rightarrow \infty$.
\end{lemma}

\begin{proof}
 $($i$)$ Note that if $Y_i(u_1) =0$ for some $i$ and $u_1$, then $Y_i(u) =0$ for all $u$.
Now (i) follows from the uniqueness of solutions of the $\operatorname{ODE}$ (\ref{eq RF torus r=1 A2a}).

 $($ii$)$ By (\ref{eq RF torus r=1 A2a}) we have
\begin{equation}
\frac{dY_i}{du} \geq Y^2_i( (n_i+1)\,q_i^2\,Y_i -2p_i). \label{eq ode compa u * finite}
\end{equation}
Let $a$ and $b$ be two positive constants.
The $\operatorname{ODE}$
\[
\frac{dz}{du} =z^2(az-b), \quad \text{with initial condition} \,\,z(u_1) >\frac{b}{a}
\]
has solution
\[
u=\frac{1}{bz} +\frac{a}{b^2} \ln \left (1-\frac{b}{az} \right ) +c
< \frac{1}{bz}  + \frac{a}{b^2} \left (-\frac{b}{az}  -  \frac{b^2}{2a^2z^2} \right ) +c
=- \frac{1}{2az^2}+c
\]
where the constant $c \in (u_1, \infty)$ is determined by the initial condition.
With $a =  (n_i+1)\, q_i^2$, $b=2p_i$, and $z(u_1) =Y_i(u_1)$,  upon comparing
(\ref{eq ode compa u * finite}) with the equation for $z$, we obtain

\begin{equation}
Y_i(u)^2 \geq z(u)^2 > \frac{1}{2a(c-u)},  \quad u \in [u_1,u_*). \label{eq Y i sol and Z sol compar}
\end{equation}
This estimate  implies  $u_* < \infty$.
By the standard extendibility theory we conclude that  $Y(u)$ approaches  $\infty$ when
$u \rightarrow u_*^-$.

From (\ref{eq RF torus r=1 A2}) we conclude that for some $i$ the positive function
$b_i(\tau)$ approaches  zero in finite $\tau$-time.
Hence the corresponding solution $g_{a,\vec{b}}(\tau)$ is not ancient.

 $($iii$)$ Since $Y(u)$ stays in the bounded set $\{ Y, \,  0 \leq Y_i \leq  \frac{2p_i}{(n_i+1)q_i^2},
 i=1, \cdots, m \}$ for all $u \in [0, u_{*})$,
 by the extendibility theory of $\rm{ODE}s$ we conclude that $u_{*} = \infty.$

We regard $Y(u)$ as a flow line of the vector field $-G(Y)$ from (\ref{eq RF torus r=1 A2a})
and consider $\omega_Y$,  its $\omega$-limit set.
Since $Y(u)$ stays in a compact set, $\omega_Y$ must be non-empty,
compact, connected, and flow-invariant (see \cite[Proposition 1.4]{PaM82} or
Theorems VII.1.1 and VII.1.2 in \cite{Hr64}).
To prove the second part of $($iii$)$ it suffices to show that if $\omega_Y$ contains a point
 in $\mathbb{R}_{>0}^m$, say  $\zeta$, then $\omega_Y=\{ \zeta \}= \{ \xi \}$.
To see this, suppose $Y(u_i) \rightarrow \zeta$ for some sequence $u_i \rightarrow \infty$,
then the monotone quantity $\bar{\lambda}(Y(u_i)) \rightarrow \bar{\lambda}(\zeta)$.
Let $Y_*(u)$ be the solution of  (\ref{eq RF torus r=1 A2})
with initial condition $Y_*(0) =\zeta$.
Then  $Y_*(u)$ is contained in $\omega_Y$ for each $u$ and  $Y_*(u)$ is the limit of $Y(u_i+ u)$.
Since by Lemma \ref{lem bar lambda monotone} function $\bar{\lambda}(Y(u))$ is monotone non-increasing,
we conclude that $\bar{\lambda}(Y_*(u))$ is a constant function.
By the equality statement in Lemma \ref{lem bar lambda monotone} we get $Y_*(u) =\xi$.
Since $\omega_Y$ is connected,  we get $\omega_Y= \{ \xi \}$.

When $\omega_Y= \{ \xi \}$, to see the corresponding solution of the Ricci flow is ancient,
 we need to prove that $\tau(u) \rightarrow \infty$ as $u \rightarrow \infty$.
By (\ref{eq RF torus r=1 A1}) $a(\tau)$ is an increasing function of $\tau$,
and it follows from the definition of $u(\tau)$ in (\ref{u-defn})  that
the inverse function $\tau(u)$ is also increasing.
 Hence fixing some $u_0>0$ and letting $\tau_0 \doteqdot \tau(u_0)$, we have
$a(\tau(u)) \geq a(\tau_0)$ for $u \in [u_0, \infty)$.
It follows from $d\tau =a(\tau(u))\, du$ that
\begin{equation}
\tau (u) -\tau_0 \geq a(\tau_0) (u-u_0), \quad u \in [u_0, \infty). \label{eq tau to infty if u}
\end{equation}
Hence $\lim_{u \rightarrow \infty} \tau(u)  =\infty$ and the functions $a(\tau)$ and $b_i(\tau)$
exist for all $\tau \in [0,\infty)$.
\end{proof}

\vskip .1cm

Now we give some sufficient conditions for the existence of ancient solutions which converge to $0$
or to $\xi$.

\begin{theorem}\label{thm flow near origin}
Let $Y(u) \in \mathbb{R}^m_{\geq 0}$ be a solution of $($\ref{eq RF torus r=1 A2a}$)$.

\smallskip
\noindent $($i$)$ There is a constant $c_0>0$ such that if $\sum_{i=1}^m Y_i(0) \leq c_0$,
then the solution $Y(u)$ satisfies $\lim_{u \rightarrow \infty} Y(u) =0$.
If $Y(u)$ further satisfies $Y(u) \in \mathbb{R}_{>0}^m$, then the corresponding  $g_{a,\vec{b}}(\tau)$
is an ancient solution of the Ricci flow on circle bundle $P_Q$.

\noindent  $($ii$)$ There is a $C^0$-family of solutions $Y(u)$ with $(m-2)$-parameters
which satisfy $\lim_{u \rightarrow \infty} Y(u) =\xi$.
Each corresponding $g_{a,\vec{b}}(\tau)$ is an ancient solution of the  Ricci flow.

\noindent  $($iii$)$  F\noindent  or each nontrivial subset $\theta \subset \{ 1, \cdots, m\}$
there is a $C^0$-family  of ancient solutions $Y(u)$ with $(|\theta| -2)$-parameters
which satisfy $\lim_{u \rightarrow \infty} Y(u) =v_{\theta}$
 and $Y_k(u) =0$ for $ k \notin \theta$. But they do not produce any ancient solution of
 the Ricci flow on $P_Q$.
\end{theorem}

\begin{proof}
$($i$)$ We assume $Y(0) \neq 0$ and  compute that
\begin{equation*}
\frac{d}{du} \sum_{i=1}^m Y_i = - 2 \sum_{i=1}^m p_i Y_i^2 +\sum_{i=1}^m q_i^2Y_i^3
+ E(Y) \sum_{i=1}^m Y_i < - \sum_{i=1}^m p_i Y_i^2 \leq - \frac{1}{m} \left (
\sum_{i=1}^m Y_i \right )^2.
\end{equation*}
The strict inequality above holds when $|Y|$ is sufficiently small
since the positive terms in the expression on the left are of third order in $Y$.
The resulting differential inequality then implies that $\lim_{u \rightarrow \infty} \sum_{j=1}^m Y_i =0$
and (i) follows.

$($ii$)$ The existence of the family follows from Hartman-Grobman theorem
(see \cite[p.59]{PaM82}) and Lemma \ref{lem newly minted Feb 20}(i).
The second part of (ii) follows from Lemma \ref{lem non ancient or ancient sol}(iii).

$($iii$)$ The existence of the family follows from Hartman-Grobman theorem
 and Lemma \ref{lem newly minted Feb 20}(ii).
Since $Y_k(u) =0$ for $ k \notin \theta$,
$Y(u)$  does not give rise to any metric $g_{a,\vec{b}}(\tau)$ on $P_Q$.
Note however that it does give rise to an ancient solution of the Ricci flow on a corresponding
circle bundle over the product of those KE factors whose indices lie in $\theta$.
\end{proof}

If we pass to the metric tensors of the ancient solutions in Theorem \ref{thm flow near origin}(i)
and (ii), then we gain an extra parameter coming from the initial value of $a(\tau)$ (cf the statement
of the Main Theorem in the Introduction). The asymptotic behavior of these ancient solutions
as $\tau \rightarrow \infty$ is described by

\begin{theorem}\label{lem growth a bi at 0}
 Let  $g_{a, \vec{b}} (\tau)$ be the ancient solution of the Ricci flow on the circle bundle $P_Q$ given by
Theorem \ref{thm flow near origin}$($i$)$.
Then

\smallskip
\noindent $($i$)$  for each $i$ and $\tau \in [0,\infty)$, we have
\[ \frac{2n_ip_i}{n_i+1}\, \tau +b_i(0) \leq b_i(\tau) \leq 2p_i \tau +b_i(0).
\]

\noindent $($ii$)$  $\lim_{\tau \rightarrow \infty} a(\tau)$ exists and is positive.
Geometrically, as $\tau \rightarrow \infty$, the radii of the circle fibres of the circle bundle
$P_Q$, equipped with the metric $g_{a, \vec{b}} (\tau)$, increase monotonically to a finite value.
The length-scale on the KE base factors grows like $\sqrt{\tau}$;

\noindent $($iii$)$ As $\tau \rightarrow \infty$, the rescaled metrics $ \tau^{-1} g_{a, \vec{b}}
(\tau)$ on $P_Q$ collapse $($in the Gromov-Hausdorff distance sense$)$ to the Einstein product metric
$2 \sum_i \,p_i g_i$ on the base $M_1 \times \cdots \times M_m$.
\end{theorem}

\begin{proof}
(i) For solution $Y(u)$ in Theorem \ref{thm flow near origin}(i) with $Y_i(u) >0$,
by Lemma \ref{lem non ancient or ancient sol} (ii) we have that for all $u$

\begin{equation*}
q_i^2Y_i (u) \leq \frac{2p_i}{n_i +1}, \label{eq Y i upper estimate}
\end{equation*}
 and hence by (\ref{eq RF torus r=1 A2}) we get
\begin{equation*}
\frac{2n_ip_i}{n_i+1} \leq \frac{d b_i}{d \tau}  \leq 2p_i. \label{eq b_i lower est}
\end{equation*}
(i) now follows.

\vskip .1cm
(ii)  By Theorem \ref{thm flow near origin}(i) we have $\lim_{\tau \rightarrow \infty} \frac{a(\tau)}
{b_i(\tau)} = \lim_{u \rightarrow \infty} Y(u)=0$ for each $i$.
Fixing an $\epsilon >0$ to be chosen later, there is a $\tau_0 \geq 0$ such that
$\frac{a(\tau)} {b_i(\tau)} < \epsilon$ for for each $i$ and  $\tau > \tau_0$.
By (i) we have
\begin{equation}
\frac{1}{a(\tau)} > \frac{( (2 p_i +1) \epsilon)^{-1}}{\tau}, \qquad
\text{ for } \tau \geq \max \{\tau_0, b_1(0), \cdots, b_m(0)\} \doteqdot \tau_1. \label{eq tau 0 choice}
\end{equation}

By (\ref{eq RF torus r=1 A1}) and (i) we get
\[
\frac{1}{a^2} \frac{da}{d \tau}  \leq  \sum_{i=1}^m  \frac{(n_i+1)^2 q_i^2}{4n_ip_i^2} \cdot
\frac{1}{(\tau +c_i)^2}  \quad \text{ for } \tau \geq \tau_1,
\]
where $c_i =\frac{b_i(0)(n_i+1)}{2n_ip_i}$.
Integrating this inequality over $[\tau_1,\tau]$ we get
\begin{align*}
- \frac{1}{a(\tau)}+ \frac{1}{a(\tau_1)} & \leq  \sum_{i=1}^m \frac{(n_i+1)^2 q_i^2}{4n_ip_i^2} \left (
-  \frac{1}{\tau +c_i} + \frac{1}{\tau_1 +c_i} \right ) \\
& \leq \left ( \sum_{i=1}^m \frac{(n_i+1)^2 q_i^2}{4n_ip_i^2 } \right) \left(\frac{1}{\tau_1} \right).
\end{align*}
Hence by (\ref{eq tau 0 choice})  we have that for $\tau > \tau_1$
\[
\frac{1}{a(\tau)}  \geq \left (  \frac{1}{(2 p_i +1) \epsilon}
- \sum_{i=1}^m \frac{(n_i+1)^2 q_i^2}{4n_ip_i^2 } \right ) \left( \frac{1}{\tau_1}\right).
\]
If we choose $\epsilon$ small enough, we have proved that $a(\tau)$ is a bounded increasing
function and hence  $\lim_{\tau \rightarrow \infty} a(\tau)$ is finite and positive.

(iii) By (ii) we have $\lim_{\tau \rightarrow \infty} \frac{a(\tau)}{\tau}  =0$.
Since $\lim_{u \rightarrow \infty} Y(u)=0$, by (\ref{eq RF torus r=1 A2})  we have
$\lim_{\tau \rightarrow \infty}  \frac{db_i(\tau)}{d\tau} =2p_i$.
Given any $\varepsilon >0$, there is a $\tau_0$
such that $|\frac{db_i(\tau)}{d\tau} -2p_i| \leq \varepsilon$ for $\tau >\tau_0$.
Hence for $\tau$ large enough we have
\begin{align*}
\left |\frac{ b_i(\tau)}{\tau} -2p_i  \right | &= \left |\frac{b_i(\tau_0)-2p_i \tau_0
+\int_{\tau_0}^\tau \left( \frac{db_i(s)}{ds}  -2p_i \right) ds }{\tau} \right |  \\
 & \leq  \frac{\left | b_i(\tau_0)-2p_i \tau_0 \right | }{\tau} +
 \frac{\int_{\tau_0}^\tau \left | \frac{db_i(s)}{ds}  -2p_i \right | ds}{\tau} \\
 &\leq \varepsilon + \varepsilon.
\end{align*}
Hence $\lim_{\tau \rightarrow \infty} \frac{ b_i(\tau)}{\tau} =2p_i$.
(iii) follows and the theorem is proved.
\end{proof}

\begin{theorem} \label{lem growth a bi at xi}
 Let  $g_{a, \vec{b}} (\tau)$ be the ancient solution of the Ricci flow on the circle bundle $P_Q$
 given by Theorem \ref{thm flow near origin}$($ii$)$.
Then for any  $\epsilon >0$ small enough there is a $\tau_0 >0$ such that
\begin{subequations}
\begin{align}
&   \left (2p_i -q_i^2 (\xi_i + \epsilon) \right ) (\tau - \tau_0) \leq b_{i} (\tau) - b_{i}(\tau_0)
\leq (2p_i -q_i^2 (\xi_i - \epsilon)) (\tau -\tau_0),   \label{eq m=2 metric est bi} \\
&  (E(\xi) - \epsilon) (\tau -\tau_0) \leq a(\tau) - a(\tau_0) \leq (E(\xi) +\epsilon) (\tau - \tau_0),
 \label{eq m=2 metric est a}
\end{align}
\end{subequations}
for $\tau \geq \tau_0$.
Hence
\begin{equation}
 \lim_{\tau \rightarrow \infty} \frac{a(\tau)}{\tau } = E(\xi), \qquad
\lim_{\tau \rightarrow \infty} \frac{b_{i}(\tau)}{\tau} = E(\xi) \xi_i^{-1}.
\label{eq curv asmp m2 omega 1 2}
\end{equation}
Geometrically,  as $\tau \rightarrow \infty$, the circle bundle $P_Q$, equipped with the
metric ${\tau}^{-1} g_{a, \vec{b}} (\tau)$, converges to a multiple of the Einstein metric
corresponding to $\xi$.
 \end{theorem}

\begin{proof}
Since $\lim_{u \rightarrow \infty}Y(u) = \xi$, given any $\epsilon >0$ small enough
we may choose $u_0$ such that
$0< \xi_i -\epsilon \leq Y_i (u) \leq \xi_i + \epsilon$ for each $i$ and $u \geq u_0$.
Let $\tau_0 =\tau(u_0)$, then we get by (\ref{eq RF torus r=1 A2})
that for $\tau \geq \tau_0$
\begin{equation*}
 2p_i -q_i^2 (\xi_i +\epsilon) \leq \frac{d b_{i}}{d \tau} \leq 2p_i -q_i^2 (\xi_i-\epsilon),
\end{equation*}
from which (\ref{eq m=2 metric est bi}) follows.

Since $\lim_{\tau \rightarrow \infty} \frac{da}{d \tau} =E(\xi)$, given $\epsilon >0$,
by choosing $u_0$ larger if necessary, we may assume further that
 $E(\xi) -\epsilon \leq \frac{da}{d \tau} \leq E(\xi) +\epsilon$
for $\tau \geq \tau_0$.
(\ref{eq m=2 metric est a})  now follows.
\end{proof}


\subsection{\bf Curvature properties of the ancient solutions on circle bundles}

In this subsection we consider the curvature and some other geometric properties of the
ancient solutions in Theorem \ref{thm flow near origin}(i) and (ii) near $\tau = \infty$.

\begin{theorem} \label{prop asym beha at infty}
Let $g_{a, \vec{b}} (\tau)$ be the ancient solution of the Ricci flow in Theorem
\ref{thm flow near origin}$($i$)$.

\smallskip
\noindent $($i$)$  $g_{a, \vec{b}} (\tau)$ is of type I as $\tau \rightarrow \infty$, i.e.,
 there is a constant $C<\infty$ such that for $\tau$ large
\[
\tau \cdot \sup_{x \in P_Q}|\operatorname{Rm}_{g_{a, \vec{b}} (\tau)} (x) |_{g_{a, \vec{b}}
(\tau)} \leq C.
\]
Note that in Theorem \ref{lem growth a bi at 0}$($iii$)$ we have proved the collapsing of the type I
rescaled metric $ \tau^{-1} g_{a, \vec{b}}(\tau)$ on $P_Q$.

\noindent $($ii$)$ for any $\kappa >0$, the solution $g_{a, \vec{b}} (\tau)$ is not
$\kappa$-noncollapsed at all scales.
\end{theorem}

\begin{proof}
(i) Recall that the metrics $g_{a, \vec{b}} (\tau) = a(\tau) \sigma( \cdot) \otimes \sigma (\cdot)
+ \sum_i b_i(\tau)\, g_i$ from (\ref{eq family of Riem metric r=1}) are Riemannian submersion type
metrics with totally geodesic fibres. In the rest of the proof we will drop the subscripts
in $g_{a, \vec{b}}$ in order to make notation less cumbersome. We will also take as background
metric $g_0 :=\sigma( \cdot) \otimes \sigma (\cdot) + \sum_i g_i$,
and choose a $g_0$-orthonormal basis $\{e_0,e^{(1)}_{1}, \cdots,e^{(1)}_{2n_1}, e^{(2)}_1,
\cdots, e^{(m)}_{2n_m} \}$  where $e_0$ is tangent to the fibres and $e^{(i)}_j$ are basic
horizontal lifts of tangent vectors to the $i$th factor of the base. Then the corresponding
$g (\tau)$-orthonormal basis is
\begin{equation}
\left \{ \tilde{e}_0 \doteqdot \frac{e_0}{\sqrt{a(\tau)}}, \frac{e^{(1)}_{1}}{\sqrt{b_1(\tau)}}, \cdots,
\frac{e^{(1)}_{2n_1}}{\sqrt{b_1(\tau)}}, \frac{ e^{(2)}_1}{\sqrt{b_2(\tau)}}, \cdots,
\frac{e^{(m)}_{2n_m}}{\sqrt{b_m(\tau)}} \right \}.
\label{eq g tau orthonorm basis}
\end{equation}

In analysing the curvature tensor of $g(\tau)$, we use the formulas given in \cite[p.241]{Bes87}.
Since the fibres are totally geodesic and $1$-dimensional, as in \cite{WZ90}, it suffices to consider
only the following components of $\operatorname{Rm}(g(\tau))$:
\begin{subequations}
\begin{align}
& g(R_{X,\tilde{e}_0}(Y), \tilde{e}_0) =g((\nabla_{\tilde{e}_0}A)_X Y, \tilde{e}_0)
  +g(A_X \tilde{e}_0, A_Y \tilde{e}_0), \label{eq rm 1 g a b tau} \\
&  g(R_{X,Y}(Z), \tilde{e}_0) =g((\nabla_{Z}A)_X Y, \tilde{e}_0),
  \label{eq rm 2 g a b tau} \\
&  g(R_{X,Y}(Z), W) =g^*(R^*(X,Y)Z,W) -2g(A_X Y, A_Z W) \notag \\
&   \hspace{1.5in}  +g(A_Y Z , A_X W)  - g(A_X Z, A_Y W),   \label{eq rm 3 g a b tau}
\end{align}
\end{subequations}
where $X,Y,Z,W$ are horizontal vectors, $\nabla$ is the Levi-Civita connection of $g(\tau)$,
and $R^*$ is the curvature tensor of the base metric $g^*(\tau) = \sum_{i} b_i(\tau)\, g_i$.
Furthermore,  $A$ is the  O'Neill $(2, 1)$-tensor for the Riemannian submersion, which has
the properties (a) $A_X \tilde{e}_0$ is the horizontal component of $\nabla_X \tilde{e}_0$,
and (b) $2 A_X Y$ is the vertical component of Lie bracket $[X,Y]$ (and up to a sign is
the connection form $\sigma$).

For term $g(A_XY, A_Z W)$ in (\ref{eq rm 3 g a b tau}) we compute
\begin{align*}
g(A_XY, A_Z W) =& g(A_XY, \frac{e_0}{\sqrt{a}}) +g(\frac{e_0}{\sqrt{a}}, A_Z W)  \\
=& \,\frac{1}{4a} \,g(F(X,Y),e_0) \cdot g(F(Z, W),e_0) \\
=& \,\frac{a(\tau)}{4} \left (\sum_i q_i \omega_i(X,Y) \right ) \cdot \left (\sum_j q_j
\omega_j (Z,W) \right ),
\end{align*}
where we have used \cite[9.54(c)]{Bes87} to get the second equality and $F$ is the curvature form
of our connection $\sigma$. The term $\omega_i(X,Y)$ is nonzero only if $X$ and $Y$ are tangent to $M_i$.
Hence if $X, Y, Z, W$ are among the $g(\tau)$-orthonormal basis vectors in (\ref{eq g tau orthonorm basis}),
we have
\begin{equation*}
g(A_XY, A_Z W) \sim O \left (\frac{a(\tau)}{b_i(\tau)b_j(\tau)} \right ) \sim O \left
(\tau^{-2} \right ) \quad \text{ as } \tau \rightarrow \infty.
\end{equation*}

The term $g^*(R^*(X,Y)Z,W)$ in (\ref{eq rm 3 g a b tau}) is nonzero only if $X, Y, Z, W$ are tangent to $M_k$
for some $k$. In this case it is the $(0,4)$-curvature tensor of metric $b_k(\tau)g_k$.
 Hence if $X, Y, Z, W$ occur among the $g(\tau)$-orthogonal basis in (\ref{eq g tau orthonorm basis}), we have
\begin{equation*}
g^*(R^*(X,Y)Z,W) \sim O \left (\frac{1}{b_i(\tau)} \right ) \sim O \left (\tau^{-1} \right )
\quad \text{ as } \tau \rightarrow \infty.
\end{equation*}

Let $\{\tilde{e}_l \}$ be the set of all horizontal vectors
from (\ref{eq g tau orthonorm basis}). The term $g(A_X \tilde{e}_0, A_Y \tilde{e}_0)$ in
(\ref{eq rm 1 g a b tau}) can be evaluated as follows.

\begin{align*}
g(A_X \tilde{e}_0, A_Y \tilde{e}_0) =& \frac{1}{a(\tau)} \sum_l \,g(A_X e_0, \tilde{e}_l)\cdot g(
\tilde{e}_l, A_Y e_0)  \\
=& \frac{1}{a(\tau)} \sum_l \,g(e_0,A_X \tilde{e}_l) \cdot g(e_0,A_Y  \tilde{e}_l)  \\
= &  \frac{1}{4a(\tau)} \sum_l \,g(e_0,F(X,\tilde{e}_l)) \cdot g(e_0,F(Y,  \tilde{e}_l))  \\
=& \frac{a(\tau)}{4} \sum_l \,\left ( \sum_i q_i \omega_i (X, \tilde{e}_l ) \right ) \cdot
 \left ( \sum_j \,q_j \omega_j (Y, \tilde{e}_l ) \right ) ,
\end{align*}
where we have used (\cite[9.21d]{Bes87})) to get the second equality.
The product term $\omega_i (X, \tilde{e}_l) \omega_j (Y, \tilde{e}_l)$ from the last line
is nonzero if both $X$ and  $Y$ are tangent to some $M_k$.
Hence if $X, Y$ occur among the $g(\tau)$-orthogonal
basis in (\ref{eq g tau orthonorm basis}), then
\[
g(A_X \tilde{e}_0, A_Y \tilde{e}_0)  \sim O \left (\frac{a(\tau)}{b_k^2(\tau)} \right )
 \sim O \left (\tau^{-2} \right ) \quad \text{ as } \tau \rightarrow \infty.
\]

The term $g((\nabla_Z A)_X Y, \tilde{e}_0)$ in (\ref{eq rm 2 g a b tau}) equals to zero (\cite[p.243]{WZ90}).
This is essentially the covariant derivative of $F$.

Finally the term $g((\nabla_{\tilde{e}_0}A)_X Y, \tilde{e}_0)$ in (\ref{eq rm 1 g a b tau}) can be computed as
follows (\cite[p.243]{WZ90}):
\begin{align}
g((\nabla_{\tilde{e}_0}A)_X Y, \tilde{e}_0) =& \frac{1}{a(\tau)} \,g((\nabla_{e_0}A)_X Y ,e_0 ) \notag \\
 = & \frac{1}{a(\tau)} \, \left ( -g({\mathscr L}_{e_0}X, {\mathscr L}_{e_0}Y) + g({\mathscr L}_{e_0}Y,
 {\mathscr L}_{e_0}X) \right ) \notag \\
= & \,0,  \label{eq nabla e A X Y zero}
\end{align}
where ${\mathscr L}_{e_0}$ is, up to a factor of $\frac{1}{2}$, the skew symmetric operator
corresponding to the curvature $2$-form via the metric $g$.

Hence
\[
\sup_{x \in P} \| \operatorname{Rm}_{g(\tau)} (x) \|_{g(\tau)}^2 =\sup_{x} \sum_{l_1,l_2,l_3,l_4}
R(\tilde{e}_{l_1}, \tilde{e}_{l_2}, \tilde{e}_{l_3}, \tilde{e}_{l_4})^2 (x) \sim  O \left (\tau^{-2} \right )
\quad \text{ as } \tau \rightarrow \infty,
\]
where we recall that $\{ \tilde{e}_{l} \}$ consists of all the $g(\tau)$-orthonormal horizontal vectors
from (\ref{eq g tau orthonorm basis}). This proves our first assertion that the ancient solution
$g_{a, \vec{b}} (\tau)$ is of type I as $\tau \rightarrow \infty$.

\vskip .1cm
(ii) By (i) and scaling  we know that the metric
$\tilde{g}(\tau) \doteqdot \frac{1}{\tau}\, g(\tau)$ has bounded
curvature for $\tau \geq 1$. On the other hand the volume is given by
\begin{align*}
\operatorname{Vol}_{\tilde{g}((\tau)}(P) = & \,\,\frac{1}{\tau^{(1+ \sum_{i=1}^m 2n_i)/2}}
\operatorname{Vol}_{g (\tau)  } (P) \\
 \sim & \,\,\tau^{-(1+ \sum_{i=1}^m 2n_i)/2} \cdot \tau^{\sum_{i=1}^m n_i}
\operatorname{Vol}_{g_0 } (P) \\
=& \,\tau^{-1/2} \operatorname{Vol}_{g_0 } (P) \rightarrow 0
\end{align*}
as $\tau \rightarrow \infty$. Hence $\tilde{g}(\tau)$ {\em cannot} be $\kappa$-noncollapsed
at all scales (uniformly in $\tau$).

 This proves the theorem.
\end{proof}

\begin{rmk} \label{Remk scalr curv form}
(i) Note that from (\ref{eq Ricci curv circle bundle}) the nonzero components of
the Ricci tensor of $g_{a, \vec{b}} (\tau)$ are given by

\begin{align*}
& \operatorname{Rc}_{g_{a, \vec{b}} (\tau)} (\tilde{e}_0,\tilde{e}_0) =\frac{1}{2}\, \sum_{i} n_iq_i^2
\frac{a(\tau)}{b_i(\tau)^2} \sim O \left (\tau^{-2} \right ),  \\
& \operatorname{Rc}_{g_{a, \vec{b}} (\tau)} \left(\frac{e_j^{(i)}}{\sqrt{b_i(\tau)}}, \frac{e_j^{(i)}}{\sqrt{b_i(
\tau)}} \right) = \frac{p_i}{b_i (\tau)} -\frac{ q_i^2}{2} \frac{a(\tau)}{b_i(\tau)^2}
\sim O \left ( \tau^{-1} \right )
\end{align*}
as $\tau \rightarrow \infty$.
By Lemma \ref{lem non ancient or ancient sol}(ii) we have $ Y_i \leq  2p_i/((n_i+1)q_i^2),\, i=1, \cdots,m$,
for any ancient solutions in  Theorem \ref{thm flow near origin}(i) and (ii).
Hence these solutions all have positive Ricci curvature over their interval of existence.
For completeness we record the formula for scalar curvature below,
\[
R_{g_{a, \vec{b}} (\tau)} = \frac{1}{a(\tau)} \sum_{i=1}^m \left (2n_i p_i \frac{a(\tau)}{b_i(\tau)}
-\frac{1}{2} n_iq_i^2 \frac{a(\tau)^2}{b_i(\tau)^2} \right ).
\]

(ii) About the sectional curvatures of $g_{a, \vec{b}} (\tau)$, note that by (\ref{eq rm 1 g a b tau})
 and (\ref{eq nabla e A X Y zero}) we know that $K(e_0 \wedge X)$ is always nonnegative for horizontal
 vectors $X$. On the other hand $K(X \wedge Y) =K^*(X \wedge Y) -3|A_XY|^2$ where $K^*$ is the sectional
 curvature of metric $g^* =\sum_i b_i(\tau)\, g_i$. So in general we can have negative sectional curvatures.
\end{rmk}

\medskip
Next we analyse the behavior of the other class of ancient solutions in Theorem \ref{thm flow near origin}.

\begin{theorem} \label{prop asym beha xi at infty}
Let $g_{a, \vec{b}} (\tau)$ be the ancient solution of the Ricci flow on circle bundle $P_Q$
given in Theorem \ref{thm flow near origin}$($ii$)$.

\smallskip
\noindent $($i$)$ The solution $g_{a, \vec{b}} (\tau)$ is of type I as $\tau \rightarrow \infty$, i.e.,
 there is a constant $C<\infty$ such that for $\tau$ large
\[
\tau \cdot \sup_{x \in P_Q}|\operatorname{Rm}_{g_{a, \vec{b}} (\tau)} (x) |_{g_{a, \vec{b}} (\tau)} \leq C.
\]
Note that in Theorem \ref{lem growth a bi at xi} we have proved that the type I rescaled metrics
$\tau^{-1} g_{a, \vec{b}}(\tau)$ on $P_Q$ converge to a multiple of the Einstein metric
 corresponding to $\xi$.

\noindent $($ii$)$ There is a $\kappa >0$ so that the solution $g_{a, \vec{b}} (\tau)$ is
$\kappa$-noncollapsed at all scales.
\end{theorem}

\begin{proof}
(i) This follows from an inspection of the proof of Theorem \ref{prop asym beha at infty}(i) and the linear
growth of $b_i(\tau)$ and $a(\tau)$ given by (\ref{eq m=2 metric est bi})--(\ref{eq m=2 metric est a}).

(ii) Since $P_Q$ is compact and $g_{a, \vec{b}}(\tau)$ is non-flat metric, there is a $\kappa_1 >0$ such
that metrics $g_{a, \vec{b}}(\tau)$ is $\kappa_1$-noncollapsed for $\tau \in [0,1]$.
Since $\tau^{-1} g_{a, \vec{b}}(\tau)$ converges to an Einstein metric, there is a $v_0 >0$ such that
$\operatorname{Vol}_{\tau^{-1} g_{a, \vec{b}}(\tau)}(B_{\tau^{-1} g_{a, \vec{b}}(\tau)}(x,1))
 \geq v_0$ for all $x \in P_Q$ and $\tau \geq 1$.
Since the Einstein metric is non-flat, there is a $\kappa_2 >0$ such that $\tau^{-1} g_{a, \vec{b}}(\tau)$
is $\kappa_2$-noncollapsed  for $\tau \geq 1$.
Since $\kappa$-noncollaping is a property preserved by scaling, hence  $g_{a, \vec{b}}(\tau)$
is $\kappa_2$-noncollapsed for $\tau \geq 1$.
Now (ii) follows.
\end{proof}

\begin{rmk} \label{cohomogeneity}
Recall that the cohomogeneity of a (compact) Riemannian manifold is the codimension
of a principal orbit of the action of its isometry group. Thus for an $n$-dimensional
Riemannian manifold, cohomogeneity $n$ means that its isometry group is discrete. In
section 3 of \cite{WZ90}, the cohomogeneity of the Einstein metrics constructed there
was studied, and it was shown that, provided that the topology of the torus bundles is
sufficiently complicated in a suitable sense, then the cohomogeneity of the Einstein metrics
is equal to the sum of the cohomogeneities of the KE factors in the base. The same arguments,
especially Lemma 3.6 there, can be used to show that the same equation holds for the
cohomogeneities of our ancient solutions, i.e.,
$$ {\rm coh}(P_Q, g_{h, \vec{b}}(\tau)) = \sum_{i=1}^{m} \, {\rm coh}(M_i, g_i).$$
(Note that $b_i(\tau)g_i$ has the same isometry group as $g_i$.)

For the convenience of the reader, we give a short summary of the main ideas involved.
The isometries of the base of our torus bundles certainly lift to the total spaces because
these are isometric toral quotients of the product of the universal covers of the principal circle
bundles associated to the anti-canonical line bundles over the base factors. The
equation above would then hold if all isometries of the total spaces
map the totally geodesic fibres to each other. Lemma 3.6 in \cite{WZ90} gives a lower bound
for the least period of all non-trivial closed curves which are projections of periodic
geodesics in the total space. Since an isometry must map closed geodesics to
closed geodesics, the assumption on the topology of the bundles ensures that least period
of the fibre geodesics is less than this lower bound. In such a situation any isometry must
be the lift of an isometry from the base or comes from the isometric torus action.

For our ancient solutions described in Theorem \ref{lem growth a bi at 0}, the metrics
$\tau^{-1} g_{a, \vec{b}}(\tau)$ automatically have very short fibres when $\tau$ is large
and so the arguments sketched above apply without any assumption on the topology of the bundles.
(Note that the bounds needed for applying Lemma 3.6 in \cite{WZ90} follow from our bounds
for $q_iY(u)$.) Since an overall scaling does not affect cohomogeneities, the equality between
cohomogeneities applies to the metrics $g_{a, \vec{b}}(\tau)$ also for large enough $\tau$.

For our ancient solutions described in Theorem \ref{lem growth a bi at xi}, when $\tau$ is
sufficiently large, the metrics $\tau^{-1} g_{a, \vec{b}}(\tau)$ are close to the Einstein metric,
and so with the same assumptions on the topology of the bundle, the desired equation
for the cohomogeneities hold.

We are then in a position to apply Theorem 1.2 in  \cite{Ko10} on the stability of
isometry groups under the Ricci flow to get the above equation for all values of $\tau$.
Therefore, our ancient solutions in general can have arbitrary cohomogeneity by suitable choices
of the Fano KE factors and, if necessary, the topology of the bundles.
\end{rmk}

\subsection{\bf The forward limits of some ancient solutions on  circle bundles}

In this subsection we investigate when ancient solutions from Theorem \ref{thm flow near origin}(i)
have the property that $\lim_{u \rightarrow -\infty} Y(u) =\xi$.
Actually such ancient solutions lie in $\Omega_+$.
In \S \ref{subsec noncollapsing anc sol m2} we will consider some ancient solutions
from Theorem \ref{thm flow near origin}(ii) for which $\lim_{u \rightarrow -\infty} Y(u) =
v_{\theta}$. Note that part (i) of the following theorem has some overlap with
Theorem \ref{thm flow near origin}(i).

\begin{theorem} \label{thm circle bundle over KE manif ancient sol}
Let $a(\tau)$ and $b_i(\tau)$ be solutions of system
$($\ref{eq RF torus r=1 A1}$)$--$($\ref{eq RF torus r=1 A2}$)$.
Let $\Omega_+$ be the compact region defined in $($\ref{eq def Omega +}$)$. Then

\smallskip
\noindent  $($i$)$  for any initial data $a(0)>0$ and $(b_1(0), \cdots, b_m(0))$ which satisfy
$\left(\frac{a(0)}{b_1(0)},\cdots, \frac{a(0)}{b_m(0)} \right)  \in \Omega_+ \setminus \{0 \}$,
$a(\tau)$ and $b_i(\tau)$ exist for all $\tau$-time and satisfy $\left(\frac{a(\tau)}{b_1(\tau)},
\cdots, \frac{a(\tau)}{b_m( \tau)} \right)  \in \Omega_+ \setminus \{0 \}$ and  $\lim_{\tau \rightarrow \infty}
 \left(\frac{a(\tau)}{b_1(\tau)},\cdots, \frac{a(\tau)}{b_m( \tau)} \right)  =0$;

\noindent  $($ii$)$  There is only one solution in $($i$)$ whose corresponding $Y(u)$
 satisfies $\lim_{u \rightarrow -\infty} Y(u)$  $ =\xi$.
For this solution $lim_{u \rightarrow -\infty} \tau(u) \doteqdot -T_1$ is finite,
where $\tau(u)$ is  defined by $($\ref{u-defn}$)$.
\end{theorem}

\begin{proof}
By the discussion before Remark \ref{Einstein-point} we may work with the system
(\ref{eq RF torus r=1 A2a}) for $Y(u)$.

$($i$)$ By assumption we may assume that the point $Y(0) \in \Omega_+$ is not
a stationary point of the vector field of (\ref{eq RF torus r=1 A2a}).
To see that the solution $Y(u)$ stays in $\Omega_+$, fix an index $ 1 \leq i_0 \leq m$.
The gradient of the level hypersurface $F_{i_0}=0$ pointing into $\Omega_+$
is given by
\begin{align*}
\nabla F_{i_0} = & (-2n_1 q_1^2Y_1, -2n_2q_2^2Y_2, \cdots ,
-2n_{i_0-1}\,q_{i_0-1}^2Y_{i_0-1}, \\
& 2p_{i_0}-2(n_{i_0}+1)\,q_{i_0}^2Y_{i_0}, -2n_{i_0+1}\,q_{i_0+1}^2Y_{i_0+1}, \cdots, -2n_mq_m^2Y_m ).
\end{align*}
Taking the inner product of $\nabla F_{i_0}$  with vector field $-G(Y)$ defined by the $\operatorname{ODE}$
(\ref{eq RF torus r=1 A2a}),
we get
\begin{align}
 & 2n_1q_1^2 Y_1^2 \left (2p_1Y_1 -  q_1^2 Y_1^2 - E(Y) \right) +\cdots
+ 2n_mq_m^2 Y_m^2 \left ( 2p_mY_m - q_m^2 Y_m^2 -E(Y) \right) \notag \\
& + (2p_{i_0}-2q_{i_0}^2 Y_{i_0}) \left ( E(Y) +  q_{i_0}^2 Y_{i_0}^2 -2p_{i_0} Y_{i_0} \right) Y_{i_0}.
\label{eq quantity for system max prin ode}
\end{align}
Because $2p_iY_i - q_i^2 Y_i^2 -E(Y) \geq 0$ in $\Omega_+$ and $ E(Y) +  q_{i_0}^2 Y_{i_0}^2
-2p_{i_0} Y_{i_0} =0$ on the hypersurface $F_{i_0}(Y)=0$, the quantity in
(\ref{eq quantity for system max prin ode}) is nonnegative and it equals zero if and only if
$Y_1F_1(Y) = \cdots =Y_mF_m(Y)=0$.

Since $Y(u)$ stays in a compact set, by the extendibility theory of $\operatorname{ODE}$s
the solution $Y(u)$ exists for all $u \in [0,\infty)$.
By Lemma \ref{lem non ancient or ancient sol}$($iii$)$ and the fact that $\Omega_+$ intersects
any coordinate hyperplane only at the origin,
the corresponding solution of the Ricci flow is either ancient with
 $\lim_{u \rightarrow \infty} Y(u) = 0 $ or converges to $ \xi $.
Below we rule out the second possibility by considering three cases.

$($ia$)$ The first case is when there is an $i_0$ such that $Y_{i_0}(0) < \xi_{i_0}$.
From $\frac{dY_{i_0}}{du}(u) =-Y_{i_0}(u) F_{i_0}(Y(u)) \leq 0$ we get
$\lim_{u \rightarrow \infty} Y_{i_0}(u) < \xi_{i_0}$.
Hence in this case we have $\lim_{u \rightarrow \infty} Y(u) =0$.

$($ib$)$ Next we rule out the case in which  $Y(0) \in \Omega_+$ and
$Y_i(0) > \xi_i$ for each $i$.
 In this situation, we have
\begin{equation}
\xi_i < \frac{p_i}{(n_i +1) q_i^2}, \qquad \text{ for all }  i. \label{eq xi uppber bdd middle}
\end{equation}
To see this we compute
\begin{align*}
 0  &\leq F_i(Y(0)) -F_i(\xi) \\
 & =  (2p_i-(n_i+1)q_i^2(Y_i(0)+\xi_i)) \cdot (Y_i(0)-\xi_i)
 - \sum_{j \neq i} n_j q_j^2 (Y_j(0)-\xi_j)(Y_j(0) + \xi_j).
\end{align*}
Hence $2p_i-(n_i+1)q_i^2(Y_i(u)+\xi_i) \geq 0$ and (\ref{eq xi uppber bdd middle}) follows.

Let
\begin{equation}
a_i \doteqdot 2(p_i -(n_i+1)q_i^2 \xi_i)  > 0 \quad \text{and} \quad b_i \doteqdot
2n_i q_i^2 \xi_i > 0. \label{ai bi def}
\end{equation}
We compute
\begin{align*}
\sum_{i=1}^m \,\frac{b_i}{a_i +b_i} & = \sum_{i=1}^m \,\frac{2n_i q_i^2 \xi_i^2}{ 2(p_i \xi_i
-(n_i+1)q_i^2 \xi_i^2) +2n_i q_i^2 \xi_i^2 } \\
& = \sum_{i=1}^m \,\frac{ 2 n_i q_i^2 \xi_i^2}{ (\sum_{j=1}^m \,n_j q_j^2 \xi_j^2) -q_i^2 \xi_i^2},
\end{align*}
where we have used $F_i(\xi)=0$ to get the last equality.
Hence we get
\begin{equation}
\sum_{i=1}^m \,\frac{b_i}{a_i +b_i} >2.  \label{eq a and b frac 2}
\end{equation}
We may therefore write $\sum_{i=1}^m \,\frac{b_i}{a_i +b_i} =\frac{1}{1-\alpha}$
for some $\alpha \in (\frac{1}{2},1)$.

We now show that there is a convex combination of the gradients $\nabla F_i(\xi)$
which is of the form $- (\delta_1, \cdots, \delta_m)$ with each $\delta_i >0$, i.e.,
there is  a  solution $(\lambda_1, \cdots, \lambda_m)$ of
\begin{equation}
\sum_{i=1}^m \lambda_i \cdot \nabla F_i(\xi)  = -(\delta_1, \cdots, \delta_m)
\label{eq convex cone normal}
\end{equation}
which satisfies $\sum_{i=1}^m \lambda_i =1$ with each $\lambda_i >0$.
Note that the equation (\ref{eq convex cone normal}) can be written in components as
\[
a_i \lambda_i- b_i \sum_{j \neq i} \lambda_j =-\delta_i,
\]
 which is equivalent to
 \[
 (a_i +b_i) \lambda_i =b_i - \delta_i.
 \]
 It is now trivial  to verify that $\delta_i =\alpha b_i$ and $\lambda_i =(1-\alpha) \frac{b_i}{a_i +b_i}$
 are solutions which satisfy all the conditions. This gives (\ref{eq convex cone normal}).

Recall that the tangent cone of the convex set $\Omega_+$ at $\xi$ is given by
\[
  T_{\xi} \Omega_{+} = \{w \in \mathbb{R}^m:\, \nabla F_i(\xi) \cdot (w-\xi) \geq 0 \}.
\]
By (\ref{eq convex cone normal}) there is a supporting plane of this cone whose inward-pointing
normal takes the form $-(\delta_1, \cdots, \delta_m)$.
We obtain an immediate contradiction that $Y(0)$ would not lie in the tangent cone,
since we have $- (\delta_1, \cdots, \delta_m) \cdot (Y(0) -\xi) <0$.

$($ic$)$ Finally we consider the case in which there are $i_0$ and $j_0$ such that $Y_{i_0}(0) = \xi_{i_0}$,
$Y_{j_0}(0) > \xi_{j_0}$, and $Y_{i}(0) \geq \xi_{i}$ for all $i$.
It follows that $F_{i_0} (Y(0)) <F_{i_0}(\xi) =0$.
Then $\frac{dY_{i_0}}{du}(0) =-Y_{i_0}(0) F_{i_0}(Y(0)) >  0$
and hence there is a small $u_1>0$ such that $Y_i(u_1) > \xi_i$ for all $i$.
Treating $u_1$ as initial time, we see that this case is impossible by the conclusion of $($ib$)$.
This finishes the proof of $($i$)$.

$($ii$)$ By Lemma \ref{lem newly minted Feb 20}(i)
 we know that exact one eigenvalue of  $\mathcal{L}_{\xi}$,
denoted say by $\lambda_1$, is negative, and that the corresponding eigenvector $z$
can be assumed to have all negative entries. Hence
$( \xi_i  \left . \nabla F_i \right |_{\xi} ) \cdot z=\lambda_1 z_i >0$ for each $i$.
Thus $ \left . \nabla F_i \right |_{\xi} \cdot z>0$ and  $z$ points into
$\operatorname{int}(\Omega_+)$.
Applying the Hartman-Grobman theorem  we conclude that one side of the
one dimensional unstable manifold of vector field $-G(Y)$ at $\xi$ lies
in $\Int(\Omega_+)$ due to $($i$)$.
This proves first part of $($ii$)$.
Note that the longtime existence of $Y(u)$ with $u \rightarrow -\infty$
follows since the solution lies in a compact set.

To see $\lim_{u \rightarrow -\infty} \tau(u)$ is finite,
since $\lim_{u \rightarrow -\infty}Y(u) =\xi$,
we can choose $u_0<0$ small enough so that when $u < u_0$ we have
$|Y_i(u) -\xi_i| < \frac{1}{2} \xi_i$ for each $i$.
From (\ref{eq RF torus r=1 A1}) we conclude
 that $\frac{da}{d \tau} > \sum_{i=1}^m \frac{1}{4} n_i q_i^2 \xi_i^2$ when $\tau < \tau(u_0)$
 which, we recall from (\ref{u-defn}), satisfies $u_0 =\int_0^{\tau(u_0)}\, \frac{1}{a(\zeta)}d \zeta$.
 This implies that there is a finite $-T_1 < \tau(u_0)$ such that $\lim_{\tau \rightarrow - T_1^+} a(\tau) =0$
 and $a(\tau) \geq (\sum_{i=1}^m \frac{1}{4} n_i q_i^2 \xi_i^2) (\tau+T_1)$ for $\tau \in [-T_1,\tau(u_0)]$.
Then it follows from
\[
\frac{du}{d \tau}= \frac{1}{a(\tau)}  \leq \frac{1}{ \sum_{i=1}^m \frac{1}{4} n_i q_i^2 \xi_i^2}
 \cdot \frac{1}{\tau +T_1} \quad \text{ for } \tau \in (-T_1,\tau(u_0)]
\]
that $\lim_{u \rightarrow -\infty} \tau(u) =-T_1$.
\end{proof}

\begin{rmk} \label{rmk m=2 circl det neg mod}
By the Hartman-Grobman theorem and Theorem \ref {thm circle bundle over KE manif ancient sol}$($ii$)$,
the other half of  the unstable manifold of the vector field $-G(Y)$ at $\xi$  lies inside
 $\Int(\Omega_-)$. Furthermore, none of the eigenvectors of  $\mathcal{L}_{\xi}$ corresponding
 to positive eigenvalues lie in the interior of the tangent cones  $T_{\xi} \Omega_+$ or
 $T_{\xi} \Omega_-$.
\end{rmk}

In the next two remarks we digress to discuss the backwards Ricci flow in pseudo Riemannian geometry.
Readers interested only in the Riemannian Ricci flow can skip to Lemma \ref{lem limt to infty}.

\begin{rmk} \label{rk psudo Riem RF ancient}
Note that in the proof of Theorem \ref{thm circle bundle over KE manif ancient sol}
what is essential for all the arguments is that we have the same sign for the products $p_i Y_i$.
By the next remark and remarks at the bottom of \cite[p.465]{ONe66} it follows that
the formulas for backwards Ricci flow (\ref{eq RF torus r=1 A1}) and
(\ref{eq RF torus r=1 A2}) hold in the pseudo Riemannian case, i.e.,
where some of $b_i, \, i \in I \subset \{1, \cdots, m\}$ and possibly $a$, are negative.

There are two situations. For convenience let us assume (by a translation in $\tau$) that
our Riemannian ancient solution is parametrized by the interval $(0, \infty)$. In the first
case, we leave $a(\tau)$ positive,  and choose a subset $I \subset \{1, \cdots, m\}$ and
replace the corresponding $b_i$ by $-b_i$. To ensure that the flow equations (\ref{eq RF torus r=1 A1})
and (\ref{eq RF torus r=1 A2}) remain unchanged we must make the corresponding $p_i, i\in I$ negative,
i.e., replace each of the Fano KE base factors $(M_i^{n_i}, g_i), i\in I$,
by a KE manifold with negative scalar curvature of the same dimension. We then obtain
ancient flows for pseudo Riemannian metrics of signature type $(1+2\sum_{i\notin I} n_i, 2\sum_{i \in I} n_i)$.

In the second situation, we replace $a(\tau)$ by $-a(\tau)$. In order to leave
(\ref{eq RF torus r=1 A1}) unchanged, we need to replace $\tau$ by $-\tau$. This means
that the corresponding Ricci flow is defined on $(0, \infty)$, i.e., it is an immortal flow.
As before choose a subset I of $I \subset \{1, \cdots, m\}$. This time, for any $i \in I$ we replace
$b_i$ by $-b_i$ and leave the corresponding Fano KE factor in the base alone, while for any
$i \notin I$, we leave $b_i$ alone and replace the corresponding Fano KE base factor by a KE
manifold with negative scalar curvature of the same dimension. Together with the reflection
in $\tau$ we see that both flow equations remain unchanged. Therefore we obtain immortal
flows for pseudo Riemannian metrics of signature type   $(2\sum_{i\notin I} n_i, 1+ 2\sum_{i \in I} n_i)$.
In particular, for the choice $I = \{1, \cdots, m\}$, we have an immortal flow with Lorentzian signature.

Actually the above correspondence for the flows applies also to the Riemannian Einstein
metrics on torus bundles over products of Fano KE manifolds and pseudo Riemannian Einstein metrics on
torus bundles over products of KE manifolds of either positive or negative scalar curvature.
\end{rmk}

\begin{rmk}
For any local coordinates $\{ x^i\}$ on a pseudo Riemannian manifold $(W^n,h)$, (using the
the convention  $R(X,Y)Z =\nabla_{[X,Y]}Z -(\nabla_X \nabla_Y Z - \nabla_Y
\nabla_X Z)$), the curvature tensor  has the same expression
in terms of $h_{ij} =h(\frac{\partial}{\partial x^i}, \frac{\partial}{\partial x^j})$, $X^i, \, Y^j$, and $Z^k$
as in the Riemannian case. (In the above $X=X^i \frac{\partial}{\partial x^i}$, $Y=Y^j \frac{\partial}{\partial x^j}$,
and $Z=Z^k \frac{\partial}{\partial x^k}$.)  Hence
the formula for the Ricci tensor
\[
\Rc(X,Z) =\tr(Y \rightarrow R(X, Y)Z) =h^{ij} h(R(X, \frac{\partial}{\partial x^i})Z, \frac{\partial}{\partial x^j})
\]
also has the same expression as in Riemannian case.

Let $\pi: (M,g) \rightarrow (B,\check{g})$ be a smooth submersion between two pseudo Riemannian manifolds.
We assume that for any $ p \in M$ the vertical space $V_p \subset (T_pM,g_p)$ is non-degenerate
and that $d \pi:T_pM \rightarrow T_{\pi(p)}B$ is an isometry on the horizontal space $H_p \doteqdot
V_p^\perp$, then the calculations for the Ricci tensor in \cite{Bes87} from
Theorem 9.28 to Proposition 9.36 go through literally without any change (as was already implied
in \cite[p.465]{ONe66}).

However if we want to use orthonormal frames to calculate the Ricci tensor for
a pseudo Riemannian submersion, we need to define an \textbf{orthonormal} frame
to be a frame $\{e_i\}$ such that  $h(e_i,e_i) \doteqdot \epsilon_i =\pm 1$ and
$h(e_i,e_j) =0$ for $i \neq j$. In such a frame the trace of a linear map
$A: T_p W \rightarrow T_p W$ is given by
\[
\tr A =\sum_i \epsilon_i h(A(e_i),e_i).
\]
The formula for the Ricci tensor of a pseudo Riemannian submersion is then
given by the same expression as in the Riemannian case in terms of the $h_{ij}$.
\end{rmk}

The technique in the proof of Theorem \ref{thm circle bundle over KE manif ancient sol}
can be used to improve  Lemma \ref{lem non ancient or ancient sol}$($ii$)$.

\begin{lemma} \label{lem limt to infty}
Let $a(\tau)$ and $b_i(\tau)$ be a solution of the system
$($\ref{eq RF torus r=1 A1}$)$--$($\ref{eq RF torus r=1 A2}$)$.
Suppose the initial data $a(0)$ and $b_i(0)$ satisfy
$Y(0) = \left(\frac{a(0)}{b_1(0)},\cdots, \frac{a(0)}{b_m(0)} \right)  \in \Omega_- \setminus
 ( \{ \xi \} \cup \{ v_\theta, \theta \in \Theta \})$.
Then   the corresponding solution $g_{a,\vec{b}}(\tau)$
terminates in  finite $\tau$-time and  $\left(\frac{a(\tau)}{b_1(\tau)},
\cdots, \frac{a(\tau)}{b_m( \tau)} \right)$ remains in $ \Omega_- \setminus
 ( \{ \xi \} \cup \{ v_\theta, \theta \in \Theta \})$.
Let $u_*$ denote the maximal $u$-time of existence of the corresponding $Y(u)$.
Then $u_{*}$ is finite and   $Y(u)$ approaches  $\infty$ as $u \nearrow u_*^-.$
\end{lemma}

\begin{proof}
Assuming  $Y(0) \in \Omega_- \setminus ( \{ \xi \} \cup \{ v_\theta,
\theta \in \Theta \})$, by an argument similar to that in the proof of
Theorem \ref{thm circle bundle over KE manif ancient sol}(i) one shows that at a
boundary point of $\partial \Omega_-$ where $F_{i_0}=0$ for some $i_0$,
the vector field  $-G(Y)$ points into $\Omega_-$.
Hence the solution $Y(u)$ stays in $\Omega_-$ as long as it exists.
Note that $Y_i(u) \neq 0$ for all $i$ and $u < u_*$.

Next we will prove that the solution $Y(u)$  approaches $\infty$ as
$u \rightarrow u_*^-$. If this does not happen, then  since $\frac{dY_i}{du} \geq 0$
for each $i$, it follows that $Y(u)$ must converge to a finite point as  $u \rightarrow u_*^-$.
Since $Y_i(0)>0$ and $Y(u)$ satisfies (\ref{eq RF torus r=1 A2a}),
the limit of $Y(u)$ must be $\xi$. There are now three cases to consider.

\smallskip

 (a) There is an $i_0$ such that $Y_{i_0}(0) > \xi_{i_0}$.
 Then $Y_{i_0}(u)$ cannot approach $\xi_{i_0}$ as $\frac{dY_{i_0}}{du} \geq 0$.
 This is a contradiction.

\smallskip

 (b) $Y_i(0) < \xi_i$ for all $i$.  In this case, since $0 \geq F_i(Y(u)) -F_i(\xi)$,
a computation similar to that in the proof of (\ref{eq xi uppber bdd middle}) gives
 $2p_i-(n_i+1)q_i^2(Y_i(u)+\xi_i) \geq 0$. Letting $u \rightarrow u_*^-$, we get
\begin{equation*}
\xi_i \leq \frac{p_i}{(n_i +1) q_i^2}, \qquad \text{ for all }  i.
\end{equation*}

Let $a_i$ and $b_i$ be defined as in (\ref{ai bi def}).  Calculating as
in (\ref{eq a and b frac 2}) we have $ \sum_{i=1}^m \,\frac{b_i}{a_i +b_i} >2$.
We may therefore write $\sum_{i=1}^m \,\frac{b_i}{a_i +b_i} =\frac{1}{1-\alpha}$
for some $\alpha \in (\frac{1}{2},1)$. Hence in this case
(\ref{eq convex cone normal}) holds again.

Let $\delta_{\min} \doteqdot \min \{ \delta_1, \cdots, \delta_m \} >0$. Then for any unit vector
$(w_1, \cdots, w_m)$ with each $w_i <0$ we have
\[
 (w_1, \cdots, w_m) \cdot (-(\delta_1, \cdots, \delta_m)) \geq \delta_{\min} \sum_{i=1}^m (-w_i)
\geq   \delta_{\min} \sum_{i=1}^m w_i^2 = \delta_{\min}.
\]
In particular taking $ (w_1, \cdots, w_m) =  \frac{Y(u) - \xi}{\| Y(u)-\xi\|}$  we have
\begin{equation}
 \frac{Y(u) - \xi}{\| Y(u)-\xi\|} \cdot (\delta_1, \cdots, \delta_m) \leq -\delta_{\min} \quad \text{ for }
 u \in [0, u_*).
   \label{eq tem we get it to win 1}
\end{equation}

By (\ref{eq convex cone normal}), for any unit vector $(w_1, \cdots, w_m)$ in the tangent
 cone $T_{\xi}\Omega_-$, we have
\[
(w_1, \cdots, w_m) \cdot (-(\delta_1, \cdots, \delta_m)) \leq  0.
\]
Since $Y(u) \in \Omega_-$ approaches $\xi$ as $u \rightarrow u_*^-$, the distance
between the unit vector $\frac{Y(u) - \xi}{\| Y(u)-\xi\|}$ and $T_{\xi}\Omega_-$
approaches zero. Hence for $u$ sufficiently close to $u_*^-$  we have
\begin{equation}
\frac{Y(u) - \xi}{\| Y(u)-\xi\|} \cdot (- (\delta_1, \cdots, \delta_m)) \leq \frac{1}{2} \delta_{\min}.
  \label{eq tem we get it to win 2}
\end{equation}
This contradicts (\ref{eq tem we get it to win 1}), hence case (b) is impossible.

\smallskip

(c) There are $i_0$ and $j_0$ such that $ Y_{i_0}(0) = \xi_{i_0}$, $Y_{j_0}(0) < \xi_{j_0}$,
and $Y_i(0) \leq \xi_i$ for all other $i$.
It follows that $F_{i_0} (Y(0)) > F_{i_0}(\xi) = 0$.
Then $\frac{dY_{i_0}}{du}(0) =-Y_{i_0}(0) F_{i_0}(Y(0)) <  0$
and hence there is a small $u_1>0$ such that $Y_i(u_1) < \xi_i$ for all $i$.
Treating $u_1$ as the initial time, we see that this case is impossible by the
conclusion of (b).

\smallskip

Now we have proved $\lim_{u \rightarrow u_*^-} Y(u) =\infty$.
 By Lemma \ref{lem non ancient or ancient sol}$($ii$)$ we conclude that the
corresponding Ricci flow solution is not ancient.
\end{proof}

Next we analyse the geometric behavior as $\tau$ tends to $-T_1$ for the special
solution given in Theorem \ref{thm circle bundle over KE manif ancient sol}$($ii$)$.

\begin{theorem}\label{lem growth a bi}
Let $g_{a, \vec{b}}(\tau), \tau \in (-T_1, \infty)$, be the ancient solution of the Ricci flow
 in Theorem \ref{thm circle bundle over KE manif ancient sol}$($ii$)$.

\smallskip
\noindent $($i$)$ We have
\[
\lim_{\tau \rightarrow - T_1^+} a(\tau) =0, \quad \lim_{\tau \rightarrow - T_1^+} b_i(\tau) =0,
\quad \lim_{\tau \rightarrow - T_1^+} \frac{a(\tau)}{b_i(\tau)} =\xi_i, \quad
\lim_{\tau \rightarrow -T_1^+} \frac{a(\tau)}{T_1 +\tau} = E(\xi).
\]
Geometrically, as $\tau \rightarrow - T_1^+$, the  circle bundles $P_Q$, equipped with the
metric $g_{a, \vec{b}} (\tau)$, collapses to a point in the Gromov-Hausdorff topology.

\noindent $($ii$)$ The solution $g_{a, \vec{b}} (\tau)$ is of type I
as $\tau \rightarrow -T_1^+$, i.e., there is a positive constant $C<\infty$ such that
for $\tau \in (-T_1, 0]$ we have
\[
(T_1+ \tau) \cdot \sup_{x \in P_Q}|\operatorname{Rm}_{g_{a, \vec{b}} (\tau)} (x) |_{g_{a,
 \vec{b}} (\tau)} \leq C.
\]

\noindent $($iii$)$  As $\tau \rightarrow -T_1^+$ the rescaled metric $ (T_1+ \tau)^{-1} g_{a, \vec{b}}
(\tau)$ on $P_Q$ converges to the Einstein metric $E(\xi) g_{1,(\xi_1^{-1}, \cdots, \xi_m^{-1})} $,
 where $g_{1,(\xi_1^{-1}, \cdots, \xi_m^{-1})} = \sigma( \cdot) \otimes \sigma (\cdot)
+ \sum_i \xi_i^{-1} \, g_i$. The metric $g_{1,(\xi_1^{-1}, \cdots, \xi_m^{-1})}$
has Ricci tensor $\operatorname{Rc} =\frac{E(\xi)}{2} \cdot g_{1,(\xi_1^{-1}, \cdots, \xi_m^{-1})}$
and  scalar curvature $R =2 \sum_{i=1}^m ( n_ip_i \xi_i ) -\frac{1}{2} E(\xi) $.
\end{theorem}

\begin{proof}
 (i) The first limit was shown in the proof of Theorem \ref{thm circle bundle over KE manif ancient sol}$($ii$)$.
 The second and third limits in (i) now follow from  $\lim_{u \rightarrow -\infty} Y_i(u) =\xi_i$.
By L'Hopital's rule we have
\[
\lim_{\tau \rightarrow -T_1^+} \frac{a(\tau)}{T_1 +\tau} =\lim_{\tau \rightarrow -T_1^+} \frac{d a(\tau)}{d \tau}
= \lim_{\tau \rightarrow -T_1^+} \sum_{i=1}^m n_i q_i^2 Y_i(u(\tau))^2 = E(\xi).
\]
This proves the last equality in (i).

(ii) By (i) we have
\begin{equation} \label{eq a, bi conv T1}
\lim_{\tau \rightarrow -T_1^+} \frac{b_i(\tau)}{T_1 +\tau} =
\lim_{\tau \rightarrow - T_1^+} \left(\frac{b_i(\tau)}{a(\tau)} \cdot \frac{a(\tau)}{T_1 +\tau} \right)
=\xi_i^{-1} E(\xi).
\end{equation}
From this and the proof of Theorem \ref{prop asym beha at infty}(i), we get the desired type I curvature
estimate as $\tau \rightarrow -T_1^+$.

(iii) The convergence of the rescaled metric follows from (\ref{eq a, bi conv T1}).
That the limit is an Einstein metric  can be deduced either directly from Ricci curvature
formula in Remark \ref{Remk scalr curv form}(i) or by appealing to \cite[(1.5),(1.6)]{WZ90}.
The scalar curvature formula follows from the formula in Remark \ref{Remk scalr curv form}(i).
\end{proof}

\begin{rmk}\label{remark top type cir}
 The topological properties of the  circle bundles $P_Q$ we are considering were
first investigated in \cite{WZ90}. For dimension $7$ a more detailed study of the
topological versus differential topological properties was then made by Kreck and
Stolz \cite{KS88} at the suggestion of Wang and Ziller. For the convenience of readers
not familiar with these works, we summarize below some conclusions which are relevant
to the present article.

(i) The manifolds $P_Q$ are simply connected if the $q_i$ are pairwise relatively prime.
Let us assume this holds in all of the following. Then for each fixed odd dimension
$\geq 7$, there are (countably) infinitely many homotopy types among the $P_Q$
 as well as infinitely many homeomorphic types within a fixed integral cohomology ring
 structure. Such diverse topological behaviour already occurs when the base consists of
 two KE Fano factors which are complex projective spaces of complex dimension $> 1$.
Together with the existence results in this section, we obtain type I ancient solutions
of the Ricci flow with positive Ricci curvature on all these topologically diverse
simply connected manifolds.

(ii) In the case of dimension $7$, the work of Kreck and Stolz shows that there is a circle
bundle $P_Q$ over $\C\PP^1 \times \C\PP^2$ such that its connected sum with any of the
27 exotic $7$-spheres is again a circle bundle of the form $P_{\tilde Q}$ over
$\C\PP^1 \times \C\PP^2$. Furthermore, these bundles are pairwise homeomorphic but
not diffeomorphic. Our existence theorem gives type I ancient flows on all these
manifolds. We certainly believe that this type of phenomenon holds in higher odd dimensions
as well, but at the moment the differential classification in this generality appears
to be out of reach.

(iii) Note that the rescaled ancient solutions $g_{a, \vec{b}}(\tau)/(\tau+ T_1)$
described in Theorems \ref{thm circle bundle over KE manif ancient sol}$($ii$)$
and \ref{lem growth a bi} connect the homothety classes of Einstein metrics on $P_Q$
corresponding to $\xi$ (at $\tau = -T_1$) to the homothety class of the product Einstein
metric (at $\tau=\infty$) on the lower-dimensional manifold $M_1 \times \cdots \times M_m$.

Now for suitable choices of base manifolds and $Q$ the manifolds $P_Q$ actually become
diffeomorphic (see Proposition 2.3 in \cite{WZ90}). The simplest and
lowest dimension case occurs when $M = \C\PP^1 \times \C\PP^1 = S^2 \times S^2$.
As long as the $q_i$ are nonzero and pairwise relatively prime, the $P_Q$ are all
diffeomorphic to $S^2 \times S^3$. However, the  Einstein metrics corresponding to
$\xi_Q$ actually belong to different path components of the moduli space of
Einstein metrics on $S^2 \times S^3$ (since the Einstein constants tend to zero
if the volumes are fixed). In such situations, our ancient solutions provide
a path in the space of metrics which satisfies a geometric PDE linking these
non-isometric Einstein metrics with an Einstein metric on the base space.

(iv) Note also that the Einstein metrics on those circle bundles $P_Q$ for which
the ratios $p_i/(-q_i)$ are positive and independent of $i$ are actually Sasakian-Einstein.
(Indeed, the corresponding complex line bundles admit a complete asymptotically conical
Calabi-Yau metric \cite{WW98}.) Therefore, in this special situation one obtains examples
of $\kappa$-non-collapsed ancient solutions whose backwards limits are Sasakian-Einstein.
\end{rmk}

\subsection{$\kappa$-noncollapsed ancient solutions on  circle bundles when $m=2$ and $3$}
\label{subsec noncollapsing anc sol m2}

In this subsection we consider the ancient solutions from Theorem \ref{thm flow near origin}(ii)
for the special cases when $m=2$ and $m=3$. We will show that
$\lim_{u \rightarrow -\infty} Y(u) =v_{\theta}$ and discuss their geometric properties
as $u \rightarrow -\infty$.

First we consider the $m=2$ case with initial condition $Y(0) \in \Omega_k$ for $k=1,2$.
Then equation (\ref{eq RF torus r=1 A2a}) becomes

\begin{equation}
\frac{dY_1}{du} =-Y_1 F_1(Y), \qquad \frac{dY_2}{du} =-Y_2 F_2(Y), \label{eq evol Y m=2}
\end{equation}
where
\begin{align*}
& F_i(Y) =2p_iY_i -q_i^2Y_i^2 -E(Y), \quad i=1,2, \,\,{\rm and} \\
& E(Y) =n_1 q_1^2Y_1^2 +  n_2q_2^2Y_2^2.
\end{align*}
Note that $\Omega_k$ contains three zeros of vector field $(-Y_1F_1(Y),-Y_2 F_2(Y))$: 0, $\xi$
and $v_k$, where
$v_1 = (\frac{2p_1}{(n_1+1)q_1^2},0)$ and $v_2 =(0, \frac{2p_2}{(n_2+1)q_2^2})$.

By the linear analysis at $\xi$ given in Lemma \ref{lem newly minted Feb 20} and
Remark  \ref{rmk m=2 circl det neg mod}, we conclude that the two branches of the (local)
stable curve of $(-Y_1F_1(Y),-Y_2 F_2(Y))$ at $\xi$ lie respectively inside $\operatorname{Int}(\Omega_1)$
and $\operatorname{Int}(\Omega_2)$ and we will denote them by $\gamma_1(u)$ and $\gamma_2(u)$.
We assume that $\gamma_k(0)$ is close to $\xi$ and extend  $\gamma_k(u)$ backwards to its maximal
time of existence.

\begin{proposition} \label{prop m=2 Omega 1 2 gamma}
\noindent $($i$)$  $\lim_{u \rightarrow \infty}
\gamma_k(u) = \xi$, $ k=1,2$.

\noindent $($ii$)$ $\gamma_k(u)$ exists on $(-\infty, \infty)$ and $\lim_{u \rightarrow -\infty}
 \gamma_k(u) = v_k$, $ k=1,2$.
\end{proposition}

\begin{proof}
$($i$)$ follows from the definition of $\gamma_k$.

$($ii$)$ We regard $\gamma_k(u)$ as a flow line of the above vector field.
By arguments similar to those for  Theorem \ref{thm circle bundle over KE manif ancient sol}, one
sees that the vector field points outside of $\Omega_k$ along the boundary points. Hence going
backwards in $u$ the flow lines $\gamma_k(u)$ would stay in $\Omega_k$.
Since $\Omega_k$ is a compact set,  it follows that  $\gamma_k(u)$ is defined on all of $(-\infty, \infty)$.

We next consider $\alpha_{\gamma_k}$,  the $\alpha$-limit set of flow line $\gamma_k$.
Since $\gamma_k(u)$ stays in the compact set $\Omega_k$,  $\alpha_{\gamma_k}$ must be non-empty,
compact, connected, and flow-invariant (see e.g. \cite[Proposition 1.4]{PaM82}).
Since solution $\gamma_k(u)$ of (\ref{eq evol Y m=2}) lies in $\Omega_k$,
one easily checks that along $\gamma_k(u)$ one of $\frac{d Y_1(u)}{du}$ and $\frac{d Y_2(u)}{du} $
is $\geq 0$ and the other is $\leq 0$, and equality holds for both if and only if $Y_1F_1(Y)
=Y_2F_2(Y)=0$. Hence $\alpha$-limit set $\alpha_{\gamma_k}$ is one of $\{0 \}, \{ v_k \}, \{\xi \}$.

When $Y \in \Omega_1$ is very close to $0$, the quantity $\frac{Y_2}{Y_1}$ is very small,
hence
\begin{align*}
& \frac{d E(Y(u))}{du} =-2n_1 q_1^2 Y_1^2 F_1(Y) - 2n_2 q_2^2 Y_2^2 F_2(Y) \\
 \sim & -2n_1q_1^2 Y_1^2 \cdot (2p_1Y_1)
-2n_2q_2^2Y_2^2 \cdot (-n_1q_1^2Y_1^2) < 0,
\end{align*}
where for two quantities $A \sim B$ means $\frac{A}{B}$ is close to $1$.
This implies that $\alpha_{\gamma_1}$ cannot be $0$.
By symmetry, $\alpha_{\gamma_2}$ cannot be $0$ either.

If $\alpha_{\gamma_k} =\xi$, then $\lim_{u \rightarrow \infty} \gamma_k(u)= \lim_{u \rightarrow -\infty}
\gamma_k(u) =\xi$,
by the monotonicity of $\bar{\lambda}(\gamma_k(u))$ from Lemma \ref{lem bar lambda monotone}
we conclude that $\gamma_k(u) =\xi$ for all $u$. This is a contradiction.
Hence $\alpha_{\gamma_k}$ cannot be $\xi$ and we get (ii).
Actually by Lemma \ref{lem newly minted Feb 20} $v_k$ is a source.
\end{proof}

\begin{theorem}  \label{thm m=2 Omega 1 2 a b, tau}
Let  $g^k_{a, \vec{b}}(\tau)$ defined by functions $a_k(\tau), b_{k1}(\tau), b_{k2}(\tau)$
be the solution of the backwards Ricci flow $($\ref{eq RF torus r=1 A1}$)$ and $($\ref{eq RF torus r=1 A2}$)$
corresponding to $\gamma_k(u)$ in Proposition \ref{prop m=2 Omega 1 2 gamma}, $k=1, 2$.
Then we have the following estimates and asymptotics.

\smallskip
\noindent $($i$)$ The maximal interval of existence of  $a_k(\tau), b_{k1}(\tau), b_{k2}(\tau)$ is
$(-T_k, \infty)$ for some $T_k >0$, and $\lim_{\tau \rightarrow -T_k^+} a_k(\tau) =0$.

\noindent  $($ii$)$ For the ancient solution $g^1_{a, \vec{b}}(\tau)$ we have
\begin{subequations}
\begin{align}
&  b_{11}(0) + \left (2p_1 -q_1^2 \frac{a(0)}{b_{11}(0)} \right )\tau   \leq b_{11} (\tau)
\leq b_{11}(0) +\frac{2n_1}{n_1+1} p_1 \tau, \quad \tau \leq 0,  \label{eq m=2 metric est b1 T} \\
& b_{12}(0) + 2p_2\tau \leq b_{12}(\tau) \leq  b_{12}(0) + \left (2p_2 -q_2^2  \frac{a(0)}{b_{12}(0)}
\right ) \tau,
\quad \tau \leq 0. \label{eq m=2 metric est b2 T}
\end{align}
\end{subequations}
Hence
\begin{equation*}
 \lim_{\tau \rightarrow -T_1^+} \frac{a_1(\tau)}{T_1+\tau } = E(v_1), \quad
  \lim_{\tau \rightarrow -T_1^+} \frac{b_{11}(\tau)}{T_1+ \tau} = \frac{(n_1+1)
 q_1^2}{2p_1}E(v_1), \quad
 \lim_{\tau \rightarrow -T_1^+}  b_{12}(\tau) >0.
\end{equation*}
Geometrically,  as $\tau \rightarrow -T_1^+$, the  circle bundle $P_Q$ with two base factors,
equipped with the metric $g^1_{a, \vec{b}} (\tau)$, collapses to a multiple of the Einstein metric
 on $M_2$ in Gromov-Hausdorff topology.
 Similar conclusions can be made for the solution $g^2_{a, \vec{b}}(\tau)$.
\end{theorem}

\begin{proof}
(i) We have proved  in Theorem \ref{lem growth a bi at xi} that the solutions exist
at least on $[0, \infty)$. An argument similar to that in the proof of
Theorem \ref{thm circle bundle over KE manif ancient sol}(ii), gives
$\lim_{u \rightarrow -\infty} \tau(u)  \doteqdot -T_k > -\infty$ for the solution
 $g^k_{a, \vec{b}}(\tau)$  and $\lim_{\tau \rightarrow -T_k^+} a_k(\tau) =0$.

(ii) Since $\gamma_1(u)$ is in $\Omega_1$, we have $\frac{dY_1}{du}\leq 0$ and
$\frac{dY_2}{du} \geq 0$ along $\gamma_1(u)$.
Hence we have that for all $\tau \leq 0$
\begin{align*}
& 0< 2p_1 -q_1^2 \cdot \frac{2p_1}{(n_1+1)q_1^2} \leq \frac{d b_{11}}{d \tau} \leq 2p_1 -q_1^2 Y_1(0), \\
&  2p_2 -q_2^2 Y_2(0) \leq \frac{d b_{12}}{d \tau} \leq 2p_2,
\end{align*}
where we have used $Y_1(u) \leq \frac{2p_1}{(n_1+1)q_1^2}$ to get the first inequality above.
Integrating these inequalities we get (\ref{eq m=2 metric est b1 T}) and (\ref{eq m=2 metric est b2 T}).

 The first limit follows from
 $\lim_{\tau \rightarrow -T_1^+} \frac{da_1}{d \tau} =E(v_1)$.
 Since $\lim_{\tau \rightarrow -T_1^+} \frac{b_{11}(\tau)}{a_1(\tau)}=\lim_{u \rightarrow -\infty} $
 $\frac{1}{Y_1 (u)} = \frac{(n_1+1)q_1^2}{2p_1}$, the second limit follows from this and the first limit.

By moving the initial  condition $Y(0)$ close to $v_1$, we may assume that $\frac{4p_2 Y_2(0)}{E(v_1)} <1$
 and  that $|E(Y(u)) -E(v_1)| \leq \frac{1}{2}
E(v_1)$ for all $u \leq 0$. We have $\frac{d a_1}{d \tau} = E(Y(u)) \geq \frac{1}{2} E(v_1)$.
Integrating the inequality over $(-T_1,0]$ we get $a_1(0) \geq  \frac{1}{2} E(v_1)T_1$.
From (\ref{eq m=2 metric est b2 T}) we get
\[
b_{12}(\tau) \geq b_{12}(0) - 2p_2 T_1 \geq  b_{12}(0) - \frac {4p_2 a_1(0)}{E(v_1)}
\geq  b_{12}(0) \left (1- \frac{4p_2 Y_2(0)}{E(v_1)} \right ) >0.
\]
This proves the third limit.
\end{proof}

Now we discuss the curvature properties of the ancient solutions $g^k_{a, \vec{b}}(\tau)$
for $\tau$ close to $-T_k^+$. Note that the curvature properties for $\tau$ close to $\infty$
have been addressed in Theorem \ref{prop asym beha xi at infty}.

\begin{theorem}  \label{thm Curv property m=2 Omega 1 2 a b, tau}
Let $g^k_{a, \vec{b}}(\tau)$ be the ancient solution of the Ricci flow on a circle bundle
$P_Q$ with two base factors corresponding to $\gamma_k(u)$ in Proposition \ref{prop m=2 Omega 1 2 gamma},
$k=1, 2$. Then

\smallskip
\noindent $($i$)$ $g^k_{a, \vec{b}}(\tau)$ develops a singularity of type I as $\tau \rightarrow -T_k^+$.
By Theorem \ref{thm m=2 Omega 1 2 a b, tau}$($ii$)$ the type I rescaled metric $(T_k+\tau)^{-1}
g^k_{a, \vec{b}} (\tau)$ converges in the Gromov-Hausdorff topology to the Riemannian product
of the Einstein metric
\[
E(v_1) \left ( \sigma (\cdot) \otimes \sigma  (\cdot) + \frac{(n_1+1)q_1^2}{2p_1} g_1 \right )
\]
on the  circle bundle over $M_1$ and the Euclidean metric on $\mathbb{R}^{2n_2}$;

\noindent $($ii$)$ there is a $\kappa >0$ such that  $g^k_{a, \vec{b}}(\tau), \tau \in (-T_k, 0]$, is
$\kappa$-noncollapsed at all scales.
\end{theorem}

\begin{proof}
$($i$)$ Note that Theorem \ref{thm m=2 Omega 1 2 a b, tau}(ii) gives the asymptotic behavior of
 the functions $a_k(\tau)$ and $b_{ki}(\tau)$ as $\tau \rightarrow -T_k^+$.
An inspection of the proof of Theorem \ref{prop asym beha at infty}(i) gives
\[
(T_k+ \tau) \cdot \sup_{x \in P_Q}|\operatorname{Rm}_{g^k_{a, \vec{b}} (\tau)} (x) |_{g_{a,
 \vec{b}} (\tau)} \leq C < \infty.
 \]
The convergence of $(T_k+\tau)^{-1}g^k_{a, \vec{b}} (\tau)$ is obvious.

Following the definition of torus bundle $P_Q$ at the beginning of \S \ref{sec 2 RF eq},
the  circle bundle on $M_1$ is the bundle with Euler class $\frac{1}{2 \pi} q_1 [\omega_1]$.
Note that $F_1(Y_1) = 2p_1Y_1 -(n_1+1)q_1^2 Y_1^2$ for this  circle bundle with $m=1$.
We have $Y_1 =\frac{a}{b_1} = \frac{2p_1}{(n_1 +1)q_1^2}$ for metric
 $\sigma (\cdot) \otimes \sigma  (\cdot) + \frac{(n_1+1)q_1^2}{2p_1} g_1$.
 Since this $Y_1$ is a solution of $F_1(Y_1) =0$, this metric is Einstein by Remark \ref{Einstein-point}.

$($ii$)$ Since the Einstein metric $\sigma (\cdot) \otimes \sigma  (\cdot) + \frac{(n_1+1)q_1^2}{2p_1}
g_1$ is not flat, (ii) follows from the same argument in the proof of Theorem \ref{prop asym beha xi at infty}(ii).
\end{proof}

\vskip .1cm
Next we consider one of the $\kappa$-noncollapsed ancient solutions  from
Theorem \ref{thm flow near origin}(ii) which are defined on circle bundles $P_Q$ with $m=3$.
It follows from Theorem \ref{thm flow near origin}(i)  and Lemma \ref{lem non ancient or ancient sol}
that the corresponding flow line $Y(u)$ stays in a bounded set, and hence
it exists on $(- \infty, \infty)$.  To find the limit of  $Y(u)$ as $u \rightarrow - \infty$,
let $\alpha_Y$ be the  $\alpha$-limit set of $Y(u)$.
Then $\alpha_Y$ must be non-empty, compact, connected, and  flow-invariant.

If $\alpha_Y$ contains only  fixed points of the flow, then since $m=3$, $\alpha_Y$ must be
one of $0, v_{\{1\}}, v_{\{2\}},v_{\{3\}},$   $ v_{\{1,2\}}, v_{\{1,3\}},
v_{\{2,3\}},$ or $ \xi$.
We can ruled out $0$ since $0$ is an attractor of the flow by Theorem \ref{thm flow near origin}(i).
We can ruled out  $\xi$ by the monotonicity of $\bar{\lambda}$ as in the proof
of Proposition \ref{prop m=2 Omega 1 2 gamma}(ii).
Hence we conclude that when $\alpha_Y$ contains only fixed points,
 then $\lim_{u \rightarrow -\infty}Y(u)$ is equal to one of
 $v_{\{1\}}, v_{\{2\}},v_{\{3\}}, v_{\{12\}}, v_{\{13\}}, v_{\{23\}}$.

If $\alpha_Y$ contains points other than fixed points, then by a proof analogous to that
of Lemma \ref{lem non ancient or ancient sol}(iii)  we conclude that $\alpha_Y$ is a
subset of some coordinate plane. Without loss of generality we may assume  that
$\alpha_Y$ is a bounded flow line
$Y^\infty(u)=(Y^{\infty}_1(u), Y^{\infty}_2(u),0)$ where $Y^{\infty}_1(u) >0$
and $Y^{\infty}_2(u) \geq 0$.
By \cite[Theorem VII.1.2]{Hr64}, $Y^\infty(u)$ is defined on $(-\infty, \infty)$.
There are two cases. In the case $Y^{\infty}_2(u)=0$ for some $u$, the flow line is of the form
$(Y^{\infty}_1(u),0)$ by Lemma \ref{lem non ancient or ancient sol} (i.e., we are
actually in the  $m=1$ case) and must flow from $v_1$ to $0$ .
However the corresponding $Y^\infty(u)$ cannot be the $\alpha$-limit set of flow $Y(u)$
 since $0$ is an attractor of the flow. Therefore, we must be in the second case with
$Y^{\infty}_2(u)>0$ for all $u$. We are then reduced to the $m=2$ situation and the flow line
 $(Y^{\infty}_1(u), Y^{\infty}_2(u))$ must be one of the two flow lines described in
 Proposition \ref{prop m=2 Omega 1 2 gamma} which flows from $v_k$ to $\xi$.
But since $(v_k,0)$ is a repeller of our flow by Lemma \ref{lem newly minted Feb 20}(ii),
the corresponding $Y^\infty(u)$ cannot be the $\alpha$-limit set of flow $Y(u)$.
Hence the case of $Y^{\infty}_2(u)>0$ is also ruled out, and we have therefore
ruled out the possibility that $\alpha_Y$ contains points other than fixed points.

The argument above proves the existence of flow lines $Y(u)$ from one of
$v_{\{1\}}, v_{\{2\}},v_{\{3\}}$, $v_{\{12\}}, v_{\{13\}}, v_{\{23\}}$ to $\xi$.
Furthermore we can argue as in the proof of Theorem \ref{thm Curv property m=2 Omega 1 2 a b, tau}
and get

\begin{theorem}
When $m=3$, each ancient solution  $g_{a, \vec{b}}(\tau)$ in Theorem  \ref{thm flow near origin}$($ii$)$,
which is parametrized by the directions in a circle, develops a type I singularity at some finite time $-T < 0$.
 The rescaled metrics $(T+ \tau)^{-1} g_{a, \vec{b}}(\tau)$ converge to the Riemannian product of some
Einstein metric and a Euclidean space as $\tau \rightarrow -T^+$.
\end{theorem}

\begin{rmk}
For the ancient solutions in Theorem \ref{thm flow near origin}(ii) when $m=3$,
by dividing up the parametrizing circle, one may speculate that for each $v_{\{k\}}$,
$k=1,2,3$, there is a one-parameter family of ancient solutions $Y(u)$ which
connect $v_{\{k\}}$ to $\xi$. Moreover, these families are separated
by the three ancient solutions connecting $v_{\{12\}}$, $v_{\{13\}}$, and
$v_{\{23\}}$ with the Einstein point $\xi$.
For any $m \geq 3$, under the assumption that $p_i /|q_i|$ is independent of $i$, we
actually can verify a statement similar to the speculation.
\end{rmk}

\begin{rmk} \label{rmk Buzano part 2}
In \cite{Bo15} and \cite{BLS16} the authors studied properties of the homogeneous Ricci flow
and proved some general structure theorems about the existence and behavior of finite
time singularities. There is also an unpublished preprint \cite{Buz2} about the homogeneous
Ricci flow for compact homogeneous spaces whose isotropy representation is multiplicity free
and satisfies further technical conditions, but otherwise assumes no bound on the number of
irreducible summands.  Some of what we deduced in Theorems \ref{lem growth a bi},
\ref{thm Curv property m=2 Omega 1 2 a b, tau}, and \ref{curv--toruscase} below  about finite-time
convergence or collapse at Einstein metrics are special cases of what these authors proved when
all of the Fano KE base factors in $P_Q$ are compact homogeneous K\"ahler manifolds.
\end{rmk}

\section{\bf Ancient solutions of Ricci flow on torus bundles of rank $r > 1$} \label{sec 4 torus bundles}

Let $P_Q$ be a torus bundle of rank $r$ over a product of Fano KE manifolds
$M_1 \times \cdots \times M_m$ which satisfies the non-degeneracy  assumption
in \S \ref{sec 2 RF eq}. In this section we will consider the $r>1$ case of
the backwards Ricci flow (\ref{eq RF torus r A1}) and (\ref{eq RF torus r A2})
and establish the existence and asymptotic geometric properties of the ancient
solutions $g_{h, \vec{b}}(\tau)$ defined by (\ref{eq family of Riem metric}).
The method of proving existence is a generalization of the method used to
prove Theorem \ref{thm flow near origin}(i). We will use the notation of
\S \ref{sec 2 RF eq}.

Recall that we have fixed a decomposition of the torus $T^r
=S^1 \times \cdots \times S^1$ and hence we  have a corresponding basis
$\{e_\alpha; \alpha=1, \cdots, r\}$ for the Lie algebra $\tf$ of $T^r$.
As usual let $(h^{\alpha \beta}(\tau))_{r \times r}$
denote the inverse of matrix $(h_{\alpha \beta}(\tau))_{r \times r}$.
We begin with the following simple lemma.

\begin{lemma}\label{lem T r easy conseq}
The following estimates hold for solutions of $($\ref{eq RF torus r A1}$)$ and $($\ref{eq RF torus r A2}$)$:

\noindent $($i$)$ $b_i(\tau) \leq 2p_i \tau + b_i(0)$,

\noindent $($ii$)$ $0 \leq h^{\alpha \alpha}(\tau) \leq h^{\alpha \alpha}(0)$ and $|h^{\alpha \beta}(\tau)|
\leq C_{\alpha \beta}$ for some constant $C_{\alpha \beta}$,

\noindent $($iii$)$ there is a constant $c>0$ such that matrix $(h_{\alpha \beta}(\tau)) \geq c I_{r \times r}$.
\end{lemma}

\begin{proof} We first note that since the system $($\ref{eq RF torus r A1}$)$ and
 $($\ref{eq RF torus r A2}$)$
represent, up to a sign, a special case of the Ricci flow, the matrix $(h_{\alpha \beta}(\tau))$
is automatically positive definite as long as the flow exists at $t=-\tau$. Likewise, $b_i(\tau) > 0$.

(i) In (\ref{eq RF torus r A2}) we have $\sum_{\tilde{\alpha}, \tilde{\beta}=1}^r
 q_{\tilde{\alpha} i}q_{\tilde{\beta} i} h_{\tilde{\alpha} \tilde{\beta}}(\tau) \geq 0$ for each $i$,
 hence $\frac{d}{d \tau}b_i \leq 2p_i$ and the estimate follows.

\vskip .1cm
(ii) It follows from (\ref{eq RF torus r A1}) that
\[
 \frac{d}{d \tau}h^{\alpha \alpha} =- \sum_{i=1}^m   \frac{n_i q_{\alpha i}^2 }{b_i^2(\tau)}  \leq  0.
\]
Hence we get the first estimate in (ii). The second estimate then follows from the fact that matrix
$(h^{\alpha \beta}(\tau))$ is positive definite.

\vskip .1cm
(iii) This follows from (ii).
\end{proof}

To analyze the system (\ref{eq RF torus r A1}) and (\ref{eq RF torus r A2}) further,
we will use the dependent variables $h_{\alpha \beta}$ together with new dependent variables
\begin{equation}
\hat{Y}_i \doteqdot \frac{\hat{a}}{b_i}, \quad 1\leq i \leq m \label{eq def hat a Y var}
\end{equation}
where
\begin{equation}
\hat{a} \doteqdot \sum_{\alpha =1}^r h_{\alpha \alpha} = \tr H,
\end{equation}
and $H$ denotes the self-adjoint linear operator on $\R^r$ associated to the
symmetric tensor $h$ via the background Euclidean metric $\langle \,\, , \,\, \rangle$.
We introduce the vector $\hat{Y} = (\hat{Y}_1, \cdots, \hat{Y}_m)$
as well as the new independent variable $\hat{u}$ given by
\begin{equation}
\hat{u} =\hat{u}(\tau) = \int_0^\tau \frac{1}{\hat{a}(\zeta)}\, d \zeta. \label{eq def hat u var}
\end{equation}

Before transforming the system (\ref{eq RF torus r A1}) and (\ref{eq RF torus r A2}) to
our new variables, it is convenient to introduce in addition the vectors
\begin{equation} \label{colQ}
    Q^{(j)} \doteqdot \sum_{\alpha=1}^{r} \, q_{\alpha j}\, e_{\alpha}, \,\,\, 1 \leq j \leq m,
\end{equation}
where $\{e_1, \cdots, e_r\}$ is the standard basis of $\R^r$ from \S 2,
and the vectors
\begin{equation} \label{V-defn}
     v_{\alpha} \doteqdot \sum_{i=1}^{m} \, q_{\alpha i} \frac{\sqrt{n_i}}{b_i}\, {\tilde e}_i,
     \,\,\,\, 1 \leq \alpha \leq r,
\end{equation}
where $\{\tilde{e}_1, \cdots, \tilde{e}_m\}$ is the standard basis of $\R^m$.
The vectors $v_{\alpha}$ were first introduced in \cite{WZ90} and the matrix
$$ V \doteqdot (V_{\alpha \beta}= \langle v_{\alpha}, v_{\beta} \rangle)_{r \times r}$$
is actually positive definite since the matrix $Q$ has rank $r$ by the
non-degeneracy assumption.

Notice that the equations (\ref{eq RF torus r A1}) can now be written as
\begin{equation} \label{eq 4.6}
        \frac{dH}{d\tau} = HVH,
\end{equation}
so that
\begin{equation}  \label{fibre-eq}
       \frac{dH}{d {\hat u}} = \frac{1}{\hat a} (H (\hat{a}^2 V)H).
\end{equation}
Since ${\hat a}^2 V$ is a quadratic function of the ${\hat Y}_i$, (\ref{fibre-eq}) is an
equation involving $\hat Y$ and $H$. Moreover, using the definition of the
vectors $v_{\alpha}$, we obtain
\begin{equation}
\tr(HVH) = \sum_{i=1}^{m} \, \sum_{\alpha=1}^{r} \frac{n_i}{b_i^2} \left((HQ)_{\alpha i}\right)^2
      = \frac{1}{\hat{a}^2} \sum_{i=1}^m \sum_{\alpha=1}^r   n_i {\hat Y}_i^2
       \left((HQ)_{\alpha i}\right)^2  .   \label{eq 4.8 named}
\end{equation}
Together with
\begin{equation}
        \frac{d\hat{a}}{d \tau} = \tr(HVH),
\end{equation}
we obtain
\begin{equation}  \label{eq-ahat}
    \frac{d {\hat a}}{d {\hat u}} = \frac{1}{\hat{a}} \sum_{\alpha=1}^r \sum_{i=1}^m n_i {\hat Y}_i^2
        \left((HQ)_{\alpha i}\right)^2  .
\end{equation}

Next, noting that
\begin{equation}
     h(Q^{(i)}, Q^{(i)}) = \sum_{\tilde{\alpha}, \tilde{\beta}=1}^r \,q_{\tilde{\alpha}i} q_{\tilde{\beta}i }
                  h_{\tilde{\alpha}\tilde{\beta}},
\end{equation}
the equations (\ref{eq RF torus r A2}) can be expressed as
\begin{eqnarray} \label{eq hat Y v f}
 \frac{d \hat{Y}_i}{d\hat{u}}  & = &\hat{a} \frac{d \hat{Y}_i}{d\tau}   \notag \\
&= & \hat{Y}_i  \left ( \tr(HVH) + {\hat Y}_i  \,\frac{ h(Q^{(i)}, Q^{(i)})}{b_i}  -2p_i \hat{Y}_i \right ) \notag \\
& = &  \hat{Y}_i  \left (\frac{\tr(H ({\hat a}^2 V)H)}{{\hat a}^2} +
        \left( \frac{ h(Q^{(i)}, Q^{(i)})}{\hat{a}}\right) {\hat Y}_i^2 -2p_i \hat{Y}_i \right ).
\end{eqnarray}
Since the matrix $H(\tau) =(h_{\alpha \beta}(\tau))$ is positive definite,
we have
\begin{equation} \label{coeff-bd1}
 0 \leq \frac{h(Q^{(i)}, Q^{(i)})}{\hat a} \leq \frac{\|H\| \|Q^{(i)} \|^2 }{\hat a} \leq \|Q^{(i)} \|^2
\end{equation}
where $\| A \|$ denotes the usual matrix norm $\sqrt{\tr(A^tA)}$ (as well as the Euclidean norm, by
abuse of notation).
Also,
\begin{equation} \label{coeff-bd2}
\frac{\sum_{\alpha=1}^r ((HQ)_{\alpha i})^2}{{\hat a}^2} \leq \frac{(\|H\| \|Q^{(i)}\|)^2}{{\hat a}^2}
 \leq \|Q^{(i)}\|^2.
\end{equation}
So setting $c_i \doteqdot \|Q^{(i)}\|^2$ and using (\ref{eq 4.8 named}) we obtain the useful inequality
\begin{equation}
\frac{d \hat{Y}_i }{d\hat{u}}  \leq \hat{Y}_i \left (r \sum_{j=1}^m n_j \,c_{j} \hat{Y}_j^2
+c_{i} \hat{Y}_i^2 -2p_i \hat{Y}_i \right ).  \label{ineq tilde Y rank r tau}
\end{equation}

\begin{rmk}
The constants $c_i$ defined above clearly depend only on the topology of the bundle.
Furthermore, under the non-degeneracy assumption of $P_Q$, these constants are positive.
\end{rmk}

We next note that
\begin{equation*}
\sum_{i=1}^m 2p_i \hat{Y}_i^2 \geq \frac{2}{m} \left (\sum_{i=1}^m  \hat{Y}_i \right )^2,
\end{equation*}
and that there is a positive constant $c_{0} $ depending only on the torus
bundle (with $c_0 \geq c_i, 1 \leq i \leq m$), such that
\begin{equation}
\sum_{i=1}^m \hat{Y}_i \left (r \sum_{j=1}^m n_jc_{j} \hat{Y}_j^2
+c_{i} \hat{Y}_i^2 \right ) \leq c_0 \left (\sum_{i=1}^m \hat{Y}_i \right )^3 . \label{eq def c0 need}
\end{equation}
We introduce the quantity $\hat{E}(\hat{Y}) \doteqdot \sum_{i=1}^m  \hat{Y}_i$.
By (\ref{ineq tilde Y rank r tau}) we have
\begin{equation}
\frac{d \hat{E}}{d\hat{u}}   \leq c_0\hat{E}^3 - \frac{2}{m} \hat{E}^2
     = {\hat E}^2 \left( c_0 {\hat E} - \frac{2}{m}  \right).  \label{eq tilde Y rank r}
\end{equation}

The following result establishes the existence of  ancient solutions.

\begin{theorem}    \label{prop T r bundle over KE manif ancient sol}
Let  $P_Q$ be a torus bundle over a product of Fano KE manifolds
$M_1 \times \cdots \times M_m$ which satisfies the non-degeneracy assumption
in \S  \ref{sec 2 RF eq}. Choose an initial invariant metric
$h(0)=(h_{\alpha \beta}(0))$ on torus $T^r$ and positive constants
$b_i(0), 1 \leq i \leq m,$ which satisfy
\[
\sum_{i=1}^m \frac{\sum_{\alpha =1}^r h_{\alpha \alpha}(0)}{b_i(0)} < \frac{1}{c_0m},
\]
where $c_0$ is given in $($\ref{eq def c0 need}$)$. Let $h_{\alpha \beta}(\tau)$ and $b_i(\tau)$ be
the solution of the initial value problem of the backwards Ricci flow $($\ref{eq RF torus r A1}$)$
and $($\ref{eq RF torus r A2}$)$.  Then

\smallskip
\noindent{$($i$)$} the corresponding function $\hat{Y}(\hat{u})$ is defined for all  $\hat{u} \in [0,\infty)$,
    ${\hat Y}_i (\hat u) \leq {\hat E}(\hat u) = \sum_{i=1}^m  \hat{Y}_i \leq \frac{m}{\hat u}$ for all $i$,
     and hence $\lim_{\hat{u} \rightarrow \infty} \hat{Y}(\hat{u}) =0$;

\noindent{$($ii$)$} $h_{\alpha \beta}(\tau)$ and $b_i(\tau)$
      are defined for all  $\tau \in [0, \infty)$, i.e., the corresponding metric $g_{h, \vec{b}}(\tau)$
      gives an ancient solution of the Ricci flow on  $P_Q$;

\noindent{$($iii$)$} for each $i$ and $\tau \in [0,\infty)$,
\[
\left(2p_i - \frac{1}{m}\right)\, \tau+ b_i(0) \leq b_i(\tau) \leq  2p_i \tau + b_i(0).
\]
\end{theorem}

\begin{proof}
  We begin with the system  (\ref{fibre-eq}) and (\ref{eq hat Y v f}) of $(H(\hat{u}), \hat{Y}(\hat{u}))$
and the initial data which satisfy  $\hat{E}(\hat{Y}(0)) < \frac{1}{c_0m}$.
Let $[0, \hat{u}_*)$ be the maximal interval on which the solution exists.
Since the above system is real analytic in $H$ and $\hat{Y}$,
it follows that ${\hat Y}_i, 1 \leq i \leq m$, remain positive on $[0, \hat{u}_*)$.
Equation (\ref{eq-ahat}) shows that $\hat{a}(\hat{u})$ is an increasing function,
so $\hat{a}(\hat u)$ is positive on $[0, \hat{u}_*)$.
We can recover the independent variable $\tau$ by
\[
 \tau({\hat u}) = \int_0^{{\hat u}} \, {\hat a}(\zeta)\, d\zeta.
\]
This enable us to recover the matrix $H(\tau)$  and the functions $b_i(\tau)$
from $(H(\hat{u}), \hat{Y}(\hat{u}))$, which then satisfy
the system (\ref{eq RF torus r A1}) and (\ref{eq RF torus r A2}).
This implies that $H$ is positive definite, either by invoking properties of the Ricci flow
or by examining the $\operatorname{ODE}$ for $\det H$, which shows that $\det H$ \
is increasing in $\hat{u}$.
We have shown the correspondence between the solution $(H(\hat{u}) , \hat{Y} (\hat{u}) )$ and
solution  $H(\tau)$ and $b_i(\tau)$.

(i) and (ii) By our assumption on ${\hat E}({\hat Y}(0))$, it follows from inequality
(\ref{eq tilde Y rank r}) that $\hat{E}(\hat{Y}(\hat{u}))$ is a strictly decreasing function
on $[0,{\hat u}_*)$. In particular, we have  $\hat{Y}_i (\hat{u}) < \frac{1}{c_0m}$ on
$[0,{\hat u}_*)$ for each $i$.
 Let $\tau_* \doteqdot  \lim_{{\hat u} \rightarrow {\hat u}_*} \tau({\hat u})$.
 By Lemma \ref{lem T r easy conseq}(i)
we get
\begin{equation}
\hat{a}(0) \leq \hat{a}(\tau) =\sum_{\alpha =1}^r h_{\alpha \alpha} (\tau) \leq \frac{1}{c_0m}
 (2p_i \tau + b_i(0)) \quad \text{ for each } i \text{ and } \tau \in [0.\tau_*). \label{eq est of a hat}
\end{equation}
  Fix an index $i$.  By (\ref{eq est of a hat}) and (\ref{eq def hat u var}) we have
\[
\hat{u} (\tau) \geq \int_0^\tau \frac{c_0m}{2p_i \zeta +b_i(0) }\, d \zeta
= \frac{c_0m}{2p_i} \left ( \ln ( 2p_i \tau +b_i(0))- \ln b_i(0) \right).
\]
Hence it follows that ${\hat u}_*$ is finite if and only if ${\tau}_*$ is finite.
 We need to rule out the possibility that $\tau_*$ is finite.

 We assume below that $\tau_*$ is finite.
Since $(h_{\alpha \beta})$ is a positive definite matrix, we conclude from (\ref{eq est of a hat}) that
$|h_{\alpha \beta}(\tau)| \leq \frac{1}{c_0m}  (2p_i \tau_* + b_i(0))$ for all $\alpha$ and $\beta$.
By Lemma  \ref{lem T r easy conseq}(i)  $b_i(\tau) $ is bounded on $[0, \tau_*)$.
By the extendibility theory of $\operatorname{ODE}$ systems applied to $($\ref{eq RF torus r A1}$)$
and $($\ref{eq RF torus r A2}$)$,  solution $(H(\tau), b_i(\tau))$  can be continued beyond
time $\tau_*$, a contradiction.

  The remaining assertions in (i) follow from  integrating (\ref{eq tilde Y rank r}).
  Indeed, applying $\hat{E}(\hat{Y}(\hat{u})) <\hat{E}(\hat{Y}(0))< \frac{1}{c_0m}$
  leads to the inequality  $\frac{d{\hat E}}{d {\hat u}} \leq -\frac{1}{m} {\hat E}^2$,
which yields
$$ 0 < {\hat E}(\hat u)  \leq \frac{m}{{\hat u} + m ({\hat E}(0))^{-1}}  \leq \frac{m}{\hat u}.$$

(iii) \, The upper bound is just Lemma \ref{lem T r easy conseq}(i).
Applying the bound ${\hat Y}_i < \frac{1}{mc_0}$ to (\ref{eq RF torus r A2})
together with (\ref{coeff-bd1}), we have
\begin{equation}
\frac{d b_i}{d \tau}  \geq 2 p_i - \frac{c_i}{m c_0}  \geq 2p_i - \frac{1}{m}, \label{eq inq 4.18a}
\end{equation}
since $c_i \leq c_0$. Integrating this inequality then gives the desired lower bound.
\end{proof}

\begin{corollary} \label{cor h alpha beta bi est}
Under the same assumptions as in Theorem \ref{prop T r bundle over KE manif ancient sol}
we have

\smallskip
\noindent{$($i$)$} $\hat{a}(\tau) = \sum_{\alpha =1}^r h_{\alpha \alpha}(\tau)$ is a bounded
strictly increasing function on $[0, \infty)$;

\noindent{$($ii$)$} the metrics $h(\tau)$ converge, as $\tau \rightarrow \infty$,
to a left-invariant metric $h_*$ on $T^r$.
\end{corollary}

\begin{proof}
(i) We already know that $\hat a$ is an increasing function of $\tau$.
Combining (\ref{eq-ahat}) and (\ref{coeff-bd2}) with
Theorem \ref{prop T r bundle over KE manif ancient sol}(i)
we obtain
$$ \frac{1}{\hat a} \frac{d{\hat a}}{d {\hat u}} \leq r \, \sum_{i=1}^m \, n_i \,\frac{m^2}{{\hat u}^2}\, c_i.$$
Setting ${\hat c} = r m^2 \sum_{i=1}^m \, n_i c_i$, we see that $\hat a$ satisfies the differential inequality
\begin{equation} \label{bd-ineq}
\frac{1}{\phi} \frac{d \phi}{d\hat u} \leq {\hat c}\, {\hat u} ^{-2},
\end{equation}
whose positive solutions are bounded from above, as can be seen upon integration.
Let ${\hat a}_*$ denote the limit of ${\hat a}(\hat u)$ as $\hat u$ tends to $\infty$.

(ii) To study the convergence of the metrics $h(\tau)$, let $X$ denote an arbitrary unit
vector in $\R^r$. We shall define $h_*(X, X) \doteqdot \langle H_*(X), X \rangle$ as
$ \lim_{\hat{u}\rightarrow \infty} \langle H({\hat u})(X), X \rangle $
and then extend $h_*$ to a symmetric bilinear form in the usual way by polarization
and homothety. To justify this definition, we consider
\begin{eqnarray*}
\frac{d}{d{\hat u}} \langle H(\hat{u}) (X), X \rangle &=& \left\langle \frac{dH}{d{\hat u}}(X), X \right\rangle \\
            & = & {\hat a}(\hat u) \langle HVH(X), X \rangle \\
            & = & {\hat a}(\hat u) \langle V(HX), HX \rangle > 0,
\end{eqnarray*}
since $V$ is positive definite and ${\hat a}(\hat u) > 0$. It remains to show that
$\langle H(X), X \rangle$ is bounded from above.

Now
\begin{eqnarray*}
(HVH)_{\alpha \beta} &=& \frac{1}{{\hat a}^2} \sum_{\tilde{\alpha}, \tilde{\beta}} \sum_i \, h_{\alpha
\tilde{\alpha}} q_{\tilde{\alpha} i}
          q_{\tilde{\beta} i} n_i h_{\tilde{\beta} \beta} {\hat Y}_i^2  \\
       &\leq & \frac{1}{{\hat a}^2} \frac{m^2}{{\hat u}^2}  \sum_i \, (HQ)_{\alpha i} (HQ)_{\beta i} n_i \\
       & \leq & \frac{m^2}{{\hat u}^2} \sum_i \, c_i n_i,
\end{eqnarray*}
where we have used Theorem \ref{prop T r bundle over KE manif ancient sol}(i) and (\ref{coeff-bd2})
in the second and third line of the above computation, respectively. In other words,
$ \| HVH \| \leq {\hat c} {\hat u}^{-2} $ for some constant $\hat{c}$.
It follows that the function $\phi = \langle H(X), X \rangle$ satisfies
\[
 \frac{1}{\phi} \frac{d \phi}{d {\hat u}} \leq \frac{{\hat a}(\hat u)}{\langle H(X), X \rangle}\,
{\hat c} {\hat u}^{-2}.
\]
But ${\hat a}(\hat u)$ is  bounded from above by ${\hat a}_*$ while $\langle H(X), X \rangle \geq c >0$
 by Lemma \ref{lem T r easy conseq}(iii). So $\phi$ also satisfies a
differential inequality of type (\ref{bd-ineq}) (with a different constant)  and hence
$\langle H(X), X \rangle$ is bounded from above. It follows that the metrics $h(\tau)$ converge
 to a limit left-invariant metric.
\end{proof}

In order to examine the behavior of the curvature along the ancient solutions in Theorem
\ref{prop T r bundle over KE manif ancient sol}, we proceed as in \S 3.3,
using the fact that for each fixed $\tau$ our metrics make $P_Q$ into a Riemannian submersion
onto the base with totally geodesic fibres. Since the computations are similar we shall be
brief and only indicate the necessary changes.

Note first that any left-invariant metric on $T^r$ is automatically right-invariant
and has zero curvature. Therefore the components of the curvature tensor that need
to be computed are
$$ g(R_{X,U}(Y), V), \,\,g(R_{U,V}(X), Y), \,\,g(R_{X,Y}(Z), W)$$
where $X, Y, Z, W$ denote horizontal basic vectors and $U, V$ denote vertical
vectors. As before the components $g(R_{X,Y}(Z), U)$ are zero.

We need to choose an orthonormal basis $\{U_1, \cdots, U_r \}$ of $h(\tau)$ for
our computations. A convenient choice is to let
$$ U_{\alpha} \doteqdot \sum_{\beta =1}^r \, B^{\beta \alpha} e_{\beta}$$
where $B$ is a square root of $H$. Since $H^{-1} = (B^{-1})^2$ and
$$\tr(H^{-1}(\tau)) \leq \tr(H^{-1}(0)) \doteqdot \nu_0^2$$
as a result of the equation $\frac{d}{d\tau} H^{-1} = -V$, it follows that
$$ | B^{\alpha \beta}| \leq \|B^{-1} \| \leq \nu_0.$$
We complete $\{U_1, \cdots, U_r \}$ to an orthonormal frame for $P_Q$
by adding the adapted horizontal orthonormal basis
$\{ \tilde{e}_j^{(i)} = \frac{1}{\sqrt{b_i}}\, e_j^{(i)}, 1 \leq j \leq 2n_i, 1 \leq i \leq m\}$.

We will estimate the curvature tensor of  $g_{a, \vec{b}} (\tau)$ using the formulas
in \cite[Theorem 9.28]{Bes87}.
We now consider the terms $g(\nabla_{U_{\alpha}} A)_{X}Y, U_{\beta})$, which occur
in both $g(R_{X,U_{\alpha}}(Y), U_{\beta})$ and $g(R_{U_{\alpha},U_{\beta}}(X), Y)$.
By the computation on p. 243 of \cite{WZ90} and the fact that we have toral fibres,
it suffices to analyse terms of the form $g({\mathscr L}_{U_{\alpha}}(X), {\mathscr L}_{U_{\beta}}(Y))$
where the skew-adjoint operator ${\mathscr L}_{U}$ is defined by
$$ g({\mathscr L}(X), Y) = \frac{1}{2} \,h( \sigma(U), F(X, Y)).$$
(Recall that $\sigma, F$ are respectively the connection and curvature forms
of the bundle.) Because the KE factors in the base are orthogonal with respect
to $F$ we may let $X = \tilde{e}_k^{(i)}$ and $Y = \tilde{e}_{\ell}^{(i)}$
both of which are tangent to $M_i$. Then
\begin{eqnarray*}
g({\mathscr L}_{U_{\alpha}} (X), {\mathscr L}_{U_{\beta}}(Y)) &=&
     \frac{1}{4}\, \sum_{j=1}^{2n_i} \, h(U_{\alpha}, F(X, {\tilde e}_j^{(i)}))\, h(U_{\beta},
      F(Y, {\tilde e}_j^{(i)})) \\
     &=& \frac{1}{4 b_i} \sum_{\tilde{\alpha}, \tilde{\beta} } \sum_{j=1}^{2n_i} h(U_{\alpha}, e_{
     \tilde{\alpha} }) h(U_{\beta}, e_{\tilde{\beta}})
              \, q_{\tilde{\alpha} i}\, q_{\tilde{\beta} i} \,\omega_i(X, e_j^{(i)}) \omega_i(Y, e_j^{(i)}) \\
     &=& \frac{\delta_{k \ell}}{4b_i^2}  \sum_{\tilde{\alpha}, \tilde{\beta}} \, q_{\tilde{\alpha} i}\,
     q_{\tilde{\beta} i} \,
               h(U_{\alpha}, e_{\tilde{\alpha}}) h(U_{\beta}, e_{\tilde{\beta}}).
\end{eqnarray*}
But
\begin{equation*}
  \left| \sum_{\tilde{\alpha}, \tilde{\beta}} \, q_{\tilde{\alpha} i}\, q_{\tilde{\beta} i} \,
               h(U_{\alpha}, e_{\tilde{\alpha}}) h(U_{\beta}, e_{\tilde{\beta}}) \right|
          \leq (\nu_0)^2 \| H \|^2 r^2 \left( \sum_{\tilde{\alpha}} \, |q_{\tilde{\alpha} i}|  \right)^2,
\end{equation*}
and $\|H \|$ is bounded above by ${\hat a}_*$, so we conclude that
$$ \left| g( (\nabla_{U_{\alpha}} A)_{\tilde{e}_k^{(i)}} \tilde{e}_{\ell}^{(i)}, U_{\beta}) \right|
 \sim O(\tau^{-2}).$$

Similarly, we also have
$$ \left| g(A_{\tilde{e}_k^{(i)}} U_{\alpha}, A_{\tilde{e}_{\ell}^{(i)}} U_{\beta})\right| \sim O(\tau^{-2}) \,\,
      \mbox{and} \, \left| g(A_{X}Y, A_{Z} W)  \right|  \sim O(\tau^{-2})  $$
for $X, Y, Z, W$ chosen from the above orthonormal basis. Furthermore, as in \S 3.3,
the components of the curvature tensor of the base (with respect to the metric
$\sum_i  b_i g_i$) decay asymptotically as $O(\tau^{-1})$. Hence the norm of the
curvature tensor of our metrics decay like $O(\tau^{-1})$, as in the case of the
 circle bundles.

\begin{theorem} \label{curv--toruscase}
Under the assumptions of Theorem \ref{prop T r bundle over KE manif ancient sol} we have

\smallskip
\noindent $($i$)$ the ancient  solution $g_{h, \vec{b}} (\tau)$  on
bundle $P_Q$ is  of type I as $\tau \rightarrow \infty$,  i.e., there is a constant $C <\infty$
such that for $\tau \geq 0$
\[
\tau \cdot \sup_{x \in P_Q}|\operatorname{Rm}_{g_{h, \vec{b}} (\tau)} (x) |_{g_{h, \vec{b}} (\tau)} \leq C;
\]

\noindent $($ii$)$ as $\tau \rightarrow \infty$, the rescaled metric tensors $ \tau^{-1} g_{h, \vec{b}}(\tau)$
 on bundle $P_Q$ collapse $($in the Gromov-Hausdorff topology$)$ to the Einstein product metric
$2 \sum_i \,p_i g_i$ on the base $M_1 \times \cdots \times M_m$;

\noindent $($iii$)$  for any $\kappa >0$, the solution $g_{h, \vec{b}} (\tau)$ is not $\kappa$-noncollapsed
at all scales.

\noindent $($iv$)$ for each $\tau$, the metric $g_{h, \vec{b}} (\tau)$ has positive Ricci curvature.
\end{theorem}

\begin{proof}
Part (i) has already been proved in the discussion preceding Theorem \ref{curv--toruscase}.
The proofs of (ii) and (iii) are analogous to those for the  circle bundle case
 except that the volume decay rate for the metrics $ \tau^{-1} g_{h, \vec{b}}(\tau)$ is now ${\tau}^{-r/2}$.
See the proof of Theorem \ref{lem growth a bi at 0}(iii) and Theorem \ref{prop asym beha at infty} (ii),
 respectively.

To see (iv), from (\ref{Ricci-toral-components}) and (\ref{eq 4.6}) the toral components are
given by $\frac{1}{2} HVH$.
This is positive definite since $V$  is positive definite  by our non-degeneracy assumption.

  For the base components given in (\ref{eq Ricci curv circle bundle}),  since $\hat{Y}_i < \frac{1}{mc_0}$,
  we compute
 \begin{align*}
& \sum_{i=1}^m \left ( \frac{ p_i }{b_i(\tau)} - \sum_{\tilde{\alpha},\tilde{\beta}=1}^r \frac{1}{2}
 q_{\tilde{\alpha} i}q_{\tilde{\beta} i} \frac{ h_{\tilde{\alpha} \tilde{\beta}}(\tau)}
{b_i^2(\tau)} \right) b_i(\tau) g_i \\
=&  \sum_{i=1}^m \left (p_i  - \sum_{\tilde{\alpha},\tilde{\beta}=1}^r \frac{1}{2}
 q_{\tilde{\alpha} i}q_{\tilde{\beta} i} \frac{ h_{\tilde{\alpha} \tilde{\beta}}(\tau)}
{b_i(\tau)} \right) g_i \\
\geq &  \sum_{i=1}^m \left (p_i  -\frac{1}{2m} \right ) g_i >0,
\end{align*}
where we used the proof of (\ref{eq inq 4.18a}) to get the last inequality.
\end{proof}

\begin{rmk} \label{Ricci-scal-toruscase}
The asymptotic behaviors of the Ricci and scalar curvatures are exactly as in
the circle bundle case. For the convenience the reader we include the formula for
scalar curvature here.
\begin{eqnarray*}
 R_{g_{h, \vec{b}} (\tau)} &=& \frac{1}{\hat a} \left( 2 \sum_{i=1}^m \, n_i p_i {\hat Y}_i
       - \frac{1}{2}\, \sum_{i=1}^m \, n_i \frac{h(Q^{(i)}, Q^{(i)}) }{\hat a} {\hat Y}_i^2  \right).
\end{eqnarray*}
\end{rmk}

\begin{rmk}  \label{topology-toruscase}
As in the circle bundles case, if we take $2$-torus bundles over a product of
three complex projective spaces of different dimensions, we obtain infinitely many
homotopy types in even dimensions starting from dimension $14$. If the dimensions are
all equal but $> 1$, then the non-degenerate $2$-torus bundles have the same integral
cohomology ring but fall into infinitely many homeomorphism types. More information
can be found in \cite{WZ90}. Thus diverse topological properties are also observed
among the ancient solutions constructed in this section.
\end{rmk}

\section{Appendix: Eigenvalues of sum of a diagonal matrix and a rank one matrix} \label{sec 5 appendix}

The following algebraic lemma is used to estimate the eigenvalues of $\mathcal{L}_{\xi}$ in Lemma
\ref{lem newly minted Feb 20}. It may be the case that some experts know this fact.

\begin{lemma} \label{lem eigenvalue ai matrix and diag}
Consider the matrix
\[
A = \left [ \begin{array}{cccc}
 \epsilon_1 a_1 & 0   & \cdots  & 0  \\
0  & \epsilon_2a_2  & \cdots  & 0   \\
\vdots & \vdots & \vdots  & \vdots \\
0  & 0  & \cdots &  \epsilon_ma_m
 \end{array} \right ] +
  \left [ \begin{array}{cccc}
 a_1 & a_2   & \cdots  & a_m  \\
a_1 & a_2   & \cdots  & a_m     \\
\vdots & \vdots & \vdots  & \vdots \\
a_1 & a_2   & \cdots  & a_m
 \end{array} \right ]
\]
where each $a_i>0$ and $\epsilon_i >0$.
After a suitable permutation we may assume that $\epsilon_1 a_1 \geq \cdots \geq \epsilon_m a_m$.

\noindent{$($i$)$} If $\epsilon_ia_i$ are all distinct, then the eigenvalues $\lambda_i$ of $A$ are positive and distinct.
Assume $\lambda_1 > \cdots > \lambda_m$, then we have estimates
\begin{align}
& \epsilon_i a_i <\lambda_i < \epsilon_{i-1} a_{i-1} ~~ \text{ for } i=2, \cdots m,
    \label{eq lambda two est getby} \\
& \min_{i} \{ \epsilon_i a_i \} + \sum_{i=1}^m a_i \leq \lambda_1 \leq \max_{i} \{ \epsilon_i
a_i \} + \sum_{i=1}^m a_i.  \label{eq lambda one est Froben}
\end{align}
Furthermore, there is an eigenvector corresponding to $\lambda_1$ with all positive entries.

\noindent{$($ii$)$} Suppose that we have $\epsilon_1a_1 = \cdots = \epsilon_{k_1}a_{k_1}
\doteqdot c_1; \cdots; \epsilon_{k_1+ \cdots +k_{l-1}+1} a_{k_1+ \cdots +k_{l-1}+1} = \cdots
=\epsilon_{k_1+ \cdots +k_{l}} a_{k_1+ \cdots +k_{l}}=\epsilon_m a_m \doteqdot c_l$, where the $c_i$ are
distinct values of multiplicity $k_i$.
Let $k_0 =0$. Then  the eigenvalues of $A$ are the following: $c_j$ of multiplicity $k_j-1$ for $j=1, \cdots, l$,
 and the eigenvalues of the matrix
\[
\tilde{A} \doteqdot \left [ \begin{array}{ccc}
 \tilde{\epsilon}_1 \tilde{a}_1 & \cdots  & 0 \\
\vdots & \vdots  & \vdots \\
0   & \cdots  & \tilde{\epsilon}_l \tilde{a}_l
 \end{array} \right ] +
  \left [ \begin{array}{ccc}
 \tilde{a}_1    & \cdots  & \tilde{a}_l  \\
\vdots  & \vdots  & \vdots \\
\tilde{a}_1    & \cdots  & \tilde{a}_l
 \end{array} \right ]
 \]
where $\tilde{a}_j \doteqdot  \sum_{i=k_1+\cdots +k_{j-1}+1}^{k_1+ \cdots +k_{j}} a_i $
and $\tilde{\epsilon}_j  \doteqdot c_j/\tilde{a}_j$ for each $j$.
Note that the constants $\tilde{\epsilon}_j \tilde{a}_j$ are all distinct and the eigenvalues of $\tilde{A}$ can be
estimated using $($i$)$.
\end{lemma}

\begin{proof}
(i) Define $f(\lambda) \doteqdot \det(A-\lambda I_{m \times m})$.
 We compute
 \begin{align*}
 f(\epsilon_1a_1) = & \left | \begin{array}{cccc}
 a_1 & a_2   & \cdots  & a_m \\
a_1  & a_2 + \epsilon_2a_2- \epsilon_1a_1 & \cdots  & a_m   \\
\vdots & \vdots & \vdots  & \vdots \\
a_1  & a_2  & \cdots & a_m + \epsilon_ma_m- \epsilon_1a_1
 \end{array} \right |  \\
 =& \left | \begin{array}{cccc}
 a_1 & a_2 & \cdots  & a_m \\
0 &  \epsilon_2a_2- \epsilon_1a_1 & \cdots  & 0   \\
\vdots & \vdots & \vdots  & \vdots \\
0  & 0  & \cdots &    \epsilon_ma_m- \epsilon_1a_1
 \end{array} \right | \\
 =& a_1 \prod_{k \neq 1} (\epsilon_k a_k - \epsilon_1a_1).
 \end{align*}
 Analogously,  it follows that for any $i$ we have
 \[
(-1)^{m-i}  f(\epsilon_i a_i)=(-1)^{m-i}   a_i \prod_{k \neq i} (\epsilon_k a_k- \epsilon_ia_i) >0.
 \]
 By the intermediate value theorem, the equation
  $f(\lambda)=0$ has $m$ solutions with $\lambda_i \in (\epsilon_ia_i,\epsilon_{i-1}a_{i-1})$
 for $i=2, \cdots, m$ and $\lambda_1 >  \epsilon_1a_1$.

 To get an upper bound and a better lower bound of $\lambda_1 $, we use the
 Perron-Frobenius theorem. Since all the entries of $A$ are positive,
the largest eigenvalue $\lambda_1$ is the so-called Perron-Frobenius eigenvalue
 of $A$  whose corresponding eigenspace is one-dimensional and is spanned by an eigenvector
with positive entries. Furthermore, for any other eigenvalue of $A$ the corresponding eigenvector
has entries consisting of both signs (see [Gan59, p.64], for example, for all these properties).
It is well-known that the Perron-Frobenius eigenvalue is bounded from above and
below respectively by the maximum and minimum of row sums (see [Gan59, p.76], for example).
This implies the estimate of $\lambda_1$ in (i).

(ii)   We use the following two observations to find all the solutions of  eigenvector equation of $A$,
\begin{eqnarray*}
\left ( \left [ \begin{array}{ccc}
 c_1I_{k_1 \times k_1} & \cdots  & 0  \\
\vdots & \vdots  & \vdots \\
0   & \cdots  & c_l I_{k_l \times k_l }
 \end{array} \right ] +
  \left [ \begin{array}{ccc}
 a_1    & \cdots  & a_m  \\
\vdots  & \vdots  & \vdots \\
a_1    & \cdots  & a_m
 \end{array} \right ] \right )
  \left [ \begin{array}{c}
 u_1  \\
\vdots \\
u_m
 \end{array} \right ] =
 \lambda  \left [ \begin{array}{c}
 u_1  \\
\vdots \\
u_m
 \end{array} \right ] .
\end{eqnarray*}

{\em Observation 1.}  We assume that the column  eigenvector is of the form
 \[
 (0, \cdots, 0, u_{k_1+\cdots + k_{j-1}+1}, \cdots, u_{k_1+\cdots +k_{j-1}
 + k_{j}}, 0, \cdots,0)^T,
 \]
 where $j =1, \cdots, l$.
Without loss of generality we only work out the case when $j=1$.
Let $U_1$ be the vector $(u_1, \cdots, u_{k_1})^T$, then the eigenvector equation becomes
\begin{eqnarray*}
 \left [ \begin{array}{c}
 c_1U_1  \\
 0 \\
\vdots \\
0
 \end{array} \right ]    +
  \left [ \begin{array}{c}
 \sum_{i=1}^{k_1} a_iu_i \\
\sum_{i=1}^{k_1} a_iu_i \\
\vdots \\
\sum_{i=1}^{k_1} a_iu_i
 \end{array} \right ]    =
 \left [ \begin{array}{c}
 \lambda U_1  \\
 0\\
\vdots \\
0
 \end{array} \right ] .
\end{eqnarray*}
Hence the solutions are: eigenvalue $\lambda  =c_1$ which is of multiplicity $k_1-1$
and the corresponding eigenspace is defined by the equation $\sum_{i=1}^{k_1} a_iu_i=0$.

{\em Observation 2.}   We assume that the eigenvector satisfies condition $u_1 =\cdots =u_{k_1}
 \doteqdot w_1;  \cdots; u_{k_1+ \cdots+ k_{l-1}+1 }=\cdots=u_{k_1+ \cdots+ k_l }
  =u_{m} \doteqdot w_l$.
If we let $E_j$ denote the vector $(1, \cdots, 1)^T$ in $\mathbb{R}^{k_j}$,
then the eigenvector equation becomes
\begin{eqnarray*}
 \left [ \begin{array}{c}
 c_1w_1E_1  \\
 c_2w_2E_2 \\
\vdots \\
c_l w_l E_l
 \end{array} \right ]    +
  \left [ \begin{array}{c}
(\sum_{j=1}^l ( \sum_{i=k_1+\cdots +k_{j-1}+1}^{k_1+ \cdots +k_{j}} a_i )w_j )E_1 \\
(\sum_{j=1}^l ( \sum_{i=k_1+\cdots +k_{j-1}+1}^{k_1+ \cdots +k_{j}} a_i )w_j )E_2 \\
\vdots \\
(\sum_{j=1}^l ( \sum_{i=k_1+\cdots +k_{j-1}+1}^{k_1+ \cdots +k_{j}} a_i )w_j )E_l
 \end{array} \right ]    =
 \left [ \begin{array}{c}
 \lambda w_1E_1  \\
 \lambda w_2 E_2 \\
\vdots \\
\lambda w_l E_l
 \end{array} \right ] .
\end{eqnarray*}
The above equation is equivalent to the following eigenvector equation for
 $\lambda$ and vector $(w_1, \cdots , w_l)^T$,

\begin{eqnarray*}
\left ( \left [ \begin{array}{ccc}
 \tilde{\epsilon}_1 \tilde{a}_1 & \cdots  & 0 \\
\vdots & \vdots  & \vdots \\
0   & \cdots  & \tilde{\epsilon}_l \tilde{a}_l
 \end{array} \right ] +
  \left [ \begin{array}{ccc}
 \tilde{a}_1    & \cdots  & \tilde{a}_l  \\
\vdots  & \vdots  & \vdots \\
\tilde{a}_1    & \cdots  & \tilde{a}_l
 \end{array} \right ] \right )
  \left [ \begin{array}{c}
 w_1  \\
\vdots \\
w_m
 \end{array} \right ] =
 \lambda  \left [ \begin{array}{c}
 w_1  \\
\vdots \\
w_m
 \end{array} \right ]
\end{eqnarray*}
where $\tilde{\epsilon}_j$ and $\tilde{a}_j$ are defined in the statement of our lemma.
We may apply part (i) to get $l$ positive distinct eigenvalues for this eigenvalue problem and
hence we get  $l$ positive distinct eigenvalues of $A$ which are different from $c_1, \cdots, c_l$.

Combining the two observations together we
get $m$  linearly independent eigenvectors for $A$ since  $(k_1-1) + \cdots + (k_l-1) + l =m$.
\end{proof}


\end{document}